\documentclass{birkjour}
\usepackage{amsmath, amssymb, amsthm}

\newcommand{\R}{{\mathbf R}}
\newcommand{\Z}{{\mathbf Z}}
\newcommand{\N}{{\mathbf N}}
\newcommand{\T}{{\mathbf T}}

\newcommand    {\e}{{\mathbf x}}
\newcommand    {\by}{{\mathbf y}}

\newcommand    {\bs}{{\mathbf s}}

\newcommand    {\C}{{\mathbf C}}

\newcommand    {\bz}{{\mathbf z}}

\newtheorem{theorem}{Theorem}[section]
\newtheorem{proposition}[theorem]{Proposition}
\newtheorem{lemma}[theorem]{Lemma}
\newtheorem{corollary}[theorem]{Corollary}
\theoremstyle{definition}
\newtheorem{definition}[theorem]{Definition}
\theoremstyle{remark}
\newtheorem{remark}[theorem]{Remark}
\numberwithin{equation}{section}
\newtheorem{example}{Example}
\newtheorem{problem}{\bf Problem}[section]

\renewcommand{\theequation}{\thesection.\arabic{equation}}

\begin{document}

\title[Discreteness of the spectrum for Schr\"odinger operator]{
	Conditions for discreteness of the \\spectrum to Schr\"odinger operator via\\
	non-increasing rearrangement, Lagrangian
	\\  relaxation and perturbations}

\author[L. Zelenko]
{ Leonid Zelenko} 

\address{%
Department of Mathematics \\
University of Haifa  \\
Haifa 31905  \\
Israel}
\email{zelenko@math.haifa.ac.il}

\begin{abstract}
This work is a continuation of our previos paper \cite{Zel1}, where
for the Schr\"odinger operator $H=-\Delta+ V(\e)\cdot$ $(V(\e)\ge 0)$,  acting in the space $L_2(\R^d)\,(d\ge 3)$, some sufficient conditions for discreteness of its spectrum have been
obtained on the base of  well known Mazya -Shubin criterion and an optimization problem for a set function, which is an infinite-dimensional generalization of a binary linear programming problem. A sufficient condition for discreteness of the spectrum is formulated in terms of the non-increasing 
rearrangement of the potential $V(\e)$. Using the method of Lagrangian relaxation for this optimization problem, we obtain a sufficient condition for discreteness of the spectrum in terms of expectation and deviation of the potential. By means of suitable perturbations of the potential we obtain conditions for discreteness of the spectrum, covering  potentials which tend to infinity only on subsets of cubes, whose Lebesgue measures tend to zero when the cubes go to infinity. Also the case where the operator $H$ is defined in the space $L_2(\Omega)$ is considered ($\Omega$ is an open domain in $\R^d$).  	
\end{abstract}

\subjclass{Primary 47F05, 47B25, \\47D08, 35P05; Secondary 81Q10, 90C10, 90C27, 91A12} 	

\keywords{Schr\"odinger operator,
	discreteness of the spectrum,  \\ rearrangement of a function, optimization problem, Lagrangian \\relaxation,  perturbations.} 

\maketitle

\tableofcontents

\section{Introduction} \label{sec:introduction}
\setcounter{equation}{0}

This work is a continuation of our previous paper \cite{Zel1}.

We consider the Schr\"odinger operator $H=-\Delta+ V(\e)\cdot$,
acting in the space $L_2(\R^d)$. In what follows we
assume that $d\ge 3$, $V(\e)\ge 0$ and $V(\cdot)\in L_{1,loc}(\R^d)$. Physically $V(\e)$ is the 
potential of an external electric field. In \cite{Zel1} some sufficient conditions for discreteness of the spectrum of $H$ have been
obtained on the base of  well known Mazya -Shubin criterion \cite{M-Sh} and an optimization problem for a set function. Since also in the present paper we shall use  the Mazya -Shubin result, let us formulate it. 
Following to \cite{M-Sh}, 
consider in $\R^d$ an open domain $\mathcal{G}$ satisfying the conditions:

(a) $\mathcal{G}$ is bounded and star-shaped with respect to any point of an open ball $B_\rho(0)\,(\rho>0)$ contained in $\mathcal{G}$;

(b) $\mathrm{diam}(\mathcal{G})=2$.

As it was noticed in \cite{M-Sh}, condition (a) implies that $\mathcal{G}$ can be represented in the form
\begin{equation}\label{formofcalG}
\mathcal{G}=\{\e\in\R^d:\,\e=r\omega,\, |\omega|=1,\,0\le r<r(\omega)\}, 
\end{equation} 
where $r(\omega)$ is a positive Lipschitz function on the standard unit sphere $S^{d-1}\subset\R^d$. For $r>0$
and $\by\in\R^d$ denote 
\begin{equation}\label{dfcalGry}
\mathcal{G}_r(\by):=\{\e\in\R^d:\,r^{-1}\e\in\mathcal{G}\}+\{\by\}.
\end{equation}
Denote by
$\mathcal{N}_{\gamma}(\by,r)\;(\gamma\in(0,1))$ the collection of all
compact sets $F\subseteq\bar{\mathcal{G}}_r(\by)$ satisfying the condition
\begin{equation}\label{defNcap}
\mathrm{cap}(F)\le\gamma\,\mathrm{cap}(\bar{\mathcal G}_r(\by)),
\end{equation}
where $\mathrm{cap}(F)$ is the harmonic capacity, defined by \eqref{dfWincap}.
\begin{theorem}\label{thMazSh}[\cite{M-Sh}, Theorem 2,2]
	The spectrum of the operator $H$ is discrete, if  
	for some $r_0>0$ and for any $r\in(0,r_0)$ the condition
	\begin{equation}\label{cndmolch1}
	\lim_{\by\rightarrow\infty}\inf_{F\in\mathcal{N}_{\gamma(r)}(\by,r)}\int_{{\mathcal G}_r(\by)\setminus
		F}V(\e)\, \mathrm{d}\e=\infty,
	\end{equation}
	is satisfied, where 
	\begin{equation}\label{cndgammar}
	\forall\,r\in(0,\,r_0):\;\gamma(r)\in(0,1)\quad\mathrm{and}\quad
	\limsup_{r\downarrow 0} r^{-2}\gamma(r)=\infty,
	\end{equation}	
\end{theorem}
Let us notice that in \cite{M-Sh} the above result was formulated for the more general case, where the operator $H$ is defined in $L_2(\Omega)$ ($\Omega$ is an open domain in $\R^d$). Also a necessary condition for discreteness of the spectrum was obtained there, which is close to sufficient one. Notice that it was proved in \cite{Mol}  that the condition
\begin{equation}\label{intVtendinf} 
\forall\;r(0,\,r_0])\quad\lim_{|\by|\rightarrow\infty}\int_{Q_r(\by)}V(\e)\,\mathrm{d}\e=\infty
\end{equation}
is necessary for discreteness of the spectrum of $H$.

As we have noticed in \cite{Zel1}, condition \eqref{cndmolch1} of Theorem \ref{thMazSh} is hardly verifiable, because in order to test it, one needs to solve a difficult optimization problem, whose cost functional is the set function  $\mathcal{I}(F)=\int_{{\mathcal G}_r(\by)\setminus F}V(\e)\, \mathrm{d}\e$ and the constrain 
$F\in\mathcal{N}_{\gamma(r)}(\by,r)$ is submodular (because``cap'' is a submodular set function (definition \eqref{submod})). The goal of the previous and present papers is to obtain some easier verifiable (or at least better intuitively explainable) sufficient conditions for discreteness of the spectrum to $H$. Such kind of conditions was found in the papers \cite{Ben-Fort}, \cite{Si1}, \cite{L-S-W} and \cite{GMD}  without use of the Mazya -Shubin result. In \cite{Zel1} we have estimated the cost functional $\mathcal{I}(F)$ in \eqref{cndmolch1}  from below using the isocapacity inequality and replacing $F\in\mathcal{N}_{\gamma(r)}(\by,r)$ by a weaker but additive constrain. By this way on the base of  Theorem \ref{thMazSh} we have obtained in \cite{Zel1} some easier verifiable sufficient conditions for discreteness of the spectrum in terms of measures, which permit a reformulation in terms of non-increasing rearrangement of the potential $V(\e)$. As we have shown, these conditions are more general than ones obtained in the papers mentioned above.

In the present paper we have obtained some  sufficient conditions
for discreteness of the spectrum of $H$, using along with the arguments mentioned above also the
method of Lagrangian relaxation for the optimization problem, mentioned above. Notice that this problem is an infinite-dimensional generalization a of binary problem from integer linear programming (Problem \ref{extrprobl} and Remark \ref{rembinprob}
). The method of Lagrangian relaxations for solving  such kind of problems was applied, for instance, in the paper \cite{YRMR}. We have used some ideas from this paper. We have obtained also some result on perturbations of   the potential $V(\e)$ preserving discreteness of the spectrum of $H$, which permit to obtain conditions for discreteness of the spectrum, covering potentials that do not belong to the scope of other claims of the paper. We consider  also the case where operator $H$ is defined in the space $L_2(\Omega)$, where $\Omega$ is an open domain in $\R^d$.
 
  Let us describe briefly the main results of this paper.
 
 Theorems \ref{thcondlebesgue},  \ref{thcondlebesrear} and \ref{prlogmdensdoubrearr} were formulated and proved in the previous paper \cite{Zel1}, but in the proof of Theorem \ref{thcondlebesgue}  a technical mistake was admitted (Remark \ref{remmistake}). Therefore we give the proofs of these theorems in the present paper.
 
 Theorem \ref{thcondlebesgue} is obtained by a direct use of the isocapacity inequality. It yields a sufficient condition of discreteness of the spectrum for $H$ in terms of an optimization problem involving Lebesgue measure instead of the capacity. Let us notice that for the case $\gamma(r)\equiv const$ Theorem \ref{thcondlebesgue} was proved in \cite{Kon-Shub} (Theorem 6.1), where the operator $H$ on a Riemannian manifold with bounded geometry was considered.
 
 Theorem \ref{thcondlebesrear}  yields a sufficient condition for discreteness of the spectrum of $H$ in terms of the non-increasing rearrangement of $V(\e)$  with respect to Lebesgue measure on cubes that are  going to infinity. The sense of this theorem is following: the spectrum of $H$ is discrete, if  the normalized Lebesgue measure of subsets of cubes, where $V(\e)$ tends to infinity as the cubes go to infinity does not tend to zero too fast when the sizes of cubes tend to zero.
 
 Theorem \ref{prlogmdensdoubrearr} , based on Theorem \ref{thcondlebesrear}, yields a condition for discreteness of the spectrum for $H$
 in terms of the non-increasing  rearrangement of the potential $V(\e)$ with respect to Lebesgue measure only on cubes from $m$-adic partitions of unit cubes $Q_1(\vec l)\,(\vec l\in\Z^d)$. This circumstance enables us to construct  nontrivial examples of the Schr\"odinger operator with discrete spectrum (Example \ref{ex4}). In the formulation of this theorem we use our concept of  $(\log_m,\,\theta)$- dense system of subsets of a unit cube (Definition \ref{dfdenslogmtetpart1}).  Furthermore, in this formulation  the lower bound of the normalized Lebesgue measure of subsets of cubes from the $m$-adic partition, where $V(\e)$ tends to infinity as the cubes go to infinity, depends on sizes of these cubes. The property of $(\log_m,\,\theta)$-density ensures that on cubes with any placement and arbitrarily small sizes a similar dependence on the size is preserved for the lower bound of the normalized measure of the sets, mentioned above.  This circumstance enables us to use Theorem \ref{thcondlebesrear}  in the proof of Theorem \ref{prlogmdensdoubrearr}.

 Theorem \ref{thestLagr} yields a lower bound for optimal value of the functional in optimization Problem \ref{extrprobl} in terms of expectation and deviation of the function $W(x)$ with respect to a probability measure. In the proof of this theorem  the method of Lagrangian relaxation is used.
 
 Theorem \ref{thdiscLagr} is based on Theorem \ref{thestLagr}. The sense of the sufficient condition for discreteness of the spectrum given by this theorem is following: the spectrum of $H$ is discrete, if the rate of growth to infinity of expectation of the potential $V(\e)$ on a cube with respect to the normalized Lebesgue measure is not less than the rate of growth of its deviation while the cube goes to infinity. Let us notice that Theorem \ref{thdiscLagr} covers more narrow class of potentials than  Theorem \ref{thcondlebesrear}, since its condition limits the rate of growth of the potential in the case where $\lim_{\e\rightarrow\infty}V(\e)=+\infty$ (Example \ref{exrearbutnotLagr} and Remarks \ref{remequivcond}, \ref{remLagrimplrear}). Nevertheless,  Example \ref{exLagr} shows that it covers potentials that do not belong to the scope of criterion from \cite{GMD}, which is most resent among conditions for discreteness of the spectrum of $H$ obtained by other authors. Furthermore, in some cases the condition of Theorem \ref{thdiscLagr} is easier verifiable than the condition of Theorem \ref{thcondlebesrear} (Example \ref{exlagreasier}).
 
 Theorem \ref{thperturb} describes some perturbations $W(\e)$ of the potential $V(\e)$ preserving discreteness of the spectrum of $H$. Its proof uses upper  estimate \eqref{estintlor1} for absolute value of the quadratic form $\int_\Omega W(\e)|u(\e)|^2\,\mathrm{d}\e$, obtained in Section \ref{sec:B} (claim (ii) of Proposition \ref{lmest}). In the right hand side of this  estimate  the $L_d(\Omega)$-norm of  a vector field $\vec\Gamma(\e)$ takes part.  This field is a solution of the divergence equation $\mathrm{div}\vec\Gamma=W$ in the domain $\Omega$, This circumstance permits to make small the norm $\Vert\vec\Gamma\Vert_d$ via fast sign-alternating oscillations of $W(\e)$ (Example \ref{ex5}). Let us notice that a similar cancellation  phenomenon was discovered and used in \cite{M-V}.
 
 Theorem \ref{thFour}, based on Theorem \ref{thperturb}, yields a sufficient condition for discreteness of of the spectrum of $H$ in terms of the Fourier transform on the torus $\T^d$ for the $1$-periodic continuation (by all the variables) of the function $V\vert_{Q_1(\vec l)}\,(\vec l\in\Z^d)$ from the cube $Q_1(\vec l)$ to the whole $\R^d$. In the proof of this theorem we use Proposition \ref{prdiveq}, where the existence of a periodic potential solution of the divergence equation $\mathrm{div}\vec\Gamma=W$ is established and its $L_p$-norm $(p>2)$ is estimated.

Theorem \ref{thlattsupp} is motivated by a counterexample, constructed by A.M. Molchanov in \cite{Mol}, which shows that condition \eqref{intVtendinf} is not sufficient for  discreteness of the spectrum of $H$ for $d>1$. In this example the support of $V(\e)$ is  a union of disjoint balls, 
 whose centers form lattices on cubes $Q_1(\vec l)\,(\vec l\in\Z^d)$ and  densities of the lattices tend to infinity as $|\vec l|\rightarrow\infty$ (Remark \ref{remcountexMolch}). The following question appears: how to choose densities of lattices, radii of balls and values of the potential on balls such that the spectrum of the operator $H$ would be discrete ?  The answer of Theorem \ref{thlattsupp} is following: this occurs when condition \eqref{intVtendinf} is fulfilled with $Q_1(\vec l)$ instead of $\mathcal{G}_r(\by)$ (Remark \ref{remonint}) and density of one-dimensional sections of the lattices tends to infinity faster than values of the potential $V(\e)$ on the balls as the cubes go to infinity. Let us notice that in the counterexample of Molchanov the last condition is not fulfilled, since it contradicts the``small capacity property'' for the support of $V(\e)$. The proof of Theorem \ref{thlattsupp} is based on  Theorem \ref{thFour}. Also notice that this theorem covers such potentials that Lebesgue measures of their supports on unit cubes tend to zero as the cubes go to infinity, hence they do not belong to the scope of all the claims, based on Theorem \ref{thcondlebesgue} (Remark \ref{remoncompar}).  
 
 Theorems \ref{thmeasinstcapRVOm} and \ref{thdiscLagrOm} yield sufficient conditions for discreteness of the spectrum of $H$ in the case where $H$ is defined in  $L_2(\Omega)$ ($\Omega$ is an open domain in $\R^d$). They are generalizations of Theorems \ref{thcondlebesgue} and \ref{thdiscLagr} to this case.
 Let us notice that if in Theorem \ref{thdiscLagrOm} the Lebesgue measure of ${\mathcal G}_r(\by)\cap\Omega$ approaches from above some threshold for large $|\by|$, the condition of this theorem becomes weaker (Remark \ref{reminfldev}). On the other hand, if it is less than this threshold (i.e. the domain $\Omega$ becomes closer to bounded), then even the potential $V(\e)\equiv 0$ satisfies this condition (Corollary \ref{cordiscspecLaplace}).

  Let us notice that in \cite{T} Michael Taylor have found an alternative for the Mazya-Shubin result. His necessary and sufficient conditions for discreteness of the spectrum of $H$ are formulated in terms of the {\it scattering length } of the potential $V(\e)$ on boxes that are going to infinity. It would be interesting to extract from this result an easier verifiable sufficient condition which would be more general than results obtained in the present paper.

The paper is organized as follows. After this Introduction, in Section \ref{sec:notation}, we give basic notations, in Section \ref{sec:prel} (Preliminaries)  we introduce some concepts used in the paper. In Section \ref{sec:prevwork} we formulate some results from the previous work \cite{Zel1}, used in this paper. In Section \ref{sec:mainres} we formulate the main results of the paper and in Section \ref{sec:proofmainres} we prove them. In Section \ref{sec:examples} we recall briefly some examples, constructed in \cite{Zel1} and construct some counterexamples for claims obtained in the present paper. Sections \ref{sec:A} and \ref{sec:B} are appendices. In Section \ref{sec:A} we prove some claims on solutions of the divergence equation and  Section \ref{sec:B} yields some estimates for  $|\int_\Omega W(\e)|u(\e)|^2\,\mathrm{d}\e|$, mentioned above, by means of Sobolev Embedding Theorems. 
 
\section{Basic notations} \label{sec:notation}
\setcounter{equation}{0}

$\langle\e,\,\by\rangle\;(\e,\by\in\R^d)$ is the canonical inner product in the
real vector space $\R^d$; $|\e|=\sqrt{\langle\e\cdot\e\rangle}$ is the Euclidean
norm in $\R^d$;\vskip1mm

\noindent If $\e=(x_1,\,x_2,\,\dots,\,x_d)\in\R^d$, denote $|\e|_\infty=\max_{1\le j\le d}|x_j|$;
\vskip1mm

\noindent $\mathrm{mes}_d$ is Lebesgue measure in $\R^d$;
\vskip1mm

\noindent $\Sigma_L(\Omega)$ is the $\sigma$-algebra of all Lebesgue measurable subsets of of a domain $\Omega\subseteq\R^d$;
\vskip1mm

\noindent$\T$ is the circle (one-dimensional torus) with the length $1$;
\vskip1mm

\noindent$\T^d=\times_{k=1}^d \T$ is the $d$-dimensional torus;\vskip1mm
\vskip1mm
\noindent$\Z$ is the ring of integers; $\N$ is the set of natural numbers;
\vskip1mm

\noindent$\Z^d=\times_{k=1}^d \Z$;\vskip1mm
\vskip1mm

\noindent$[x]$ is the integer part of a real number $x$;
\vskip1mm

\noindent $\mathcal{F}(f)(\vec k)=\int_{\T^d}f(\e)\exp\big(-i\,2\pi\,\langle\vec k,\,
\e\rangle\big)\,\mathrm{d}\e\quad (\vec k\in\Z^d)$ is Fourier transform  of a complex-valued function $f(\e)$ defined on $\R^d$ and $1$-periodic by all the variables $x_1,\,x_2,\,\dots,\,x_d$ (hence $f$ can be considered as defined on $\T^d$ and it is possible to say about Fourier transform of $f $ on $\T^d$);
\vskip1mm

\noindent $\mathcal{F}^{-1}(g)(\e)=\sum_{\vec k\in\Z^d}g(\vec k)\exp\big(i\,2\pi\,\langle\vec k,\,
\e\rangle\big)\,(g:\Z^d\rightarrow\C,\,\e\in\R^d)$ is the transform inverse to $\mathcal{F}$;
\vskip1mm

\noindent $\hat f(\vec\omega):=\int_{\R^d}f(\e)\exp\big(-i2\pi\,\langle\e,\,\vec\omega\rangle\big)\,\mathrm{d}\e\,(\vec\omega\in\R^d)$ is Fourier transform on $\R^d$ of a function $f:\,\R^d\rightarrow\C$;
\vskip1mm

\noindent For Fourier transform and inverse Fourier transform of vector functions we shall use the same notations;
\vskip1mm

\noindent $\Vert f\Vert_p\,(p\in[1,\,\infty])$ is the $L_p$-norm of a function $f\in L_p(\Omega)\,(\Omega\subseteq\R^d$ and the same notation we shall use for the $L_p$-norm in $L_p(\T^d)$;
\vskip1mm

\noindent $l_p(\Z^d)\,(p\in[1,\,\infty])$ is the space of functions $g:\Z^d\rightarrow\C$ such that
\begin{equation*}
\sum_{\vec k\in\Z^d}|g(\vec k)|^p\\<\infty;
\end{equation*}
 \vskip1mm

\noindent $\Vert g\Vert_{l_p}$ is the $l_p$-norm of a function $g\in l_p(\Z^d)$;
\vskip1mm

\noindent For the analogous spaces of vector-valued functions and the norms in them we shall use the same notations;
\vskip1mm

\noindent $C_0^\infty(\Omega)$ is the collection of all functions from $C^\infty(\Omega)$ having compact supports ($\Omega$ is an open domain in $\R^d$);
\vskip1mm

\noindent $\mathrm{supp}(f)$ is the support of a function $f:\,\Omega\rightarrow\C$ ;
\vskip1mm

\noindent
$W_p^l(\Omega)\,(p\ge 1)$ is Sobolev space of functions having generalized partial derivatives up to order $l$, which belong to $L_p(\Omega)$;
\vskip1mm

\noindent$\chi_A:\;X\rightarrow\R$ is the characteristic function of a subset
$A\subseteq X$;\vskip1mm 
\vskip1mm

\noindent $B_r(\by)$ is the open ball in $\R^d$ whose radius and center are $r>0$ and $\by$;
\vskip1mm

\noindent The unit cube $Q=[0,\,1]^d\subset\R^d$ and  dilation of it $Q_r:=r\cdot Q\,(r>0)$ and translation of the latter $Q_r(\by):=Q_r+\{\by\}$.
\vskip1mm

Some specific notations will be introduced in what follows.

\section{Preliminaries} \label{sec:prel}  
\setcounter{equation}{0}

Let us come to agreement on some notations and terminology. Let $\Omega$ be an open domain in $R^d$. We denote by
$\Sigma_B(\bar\Omega)$ the $\sigma$-algebra of all
Borel subsets of $\bar\Omega$. It is known that $\Sigma_B(\bar\Omega\subseteq\Sigma_L(\bar\Omega)$.   
Let us recall the definition of the {\it harmonic (or Newtonian) capacity}\footnote[1]{In the Russian literature it is often called  \it{Wiener capacity.}}  of a compact set
$E\subset\R^d$ (\cite{M-Sh}):
\begin{eqnarray}\label{dfWincap}
&&\mathrm{cap}(E):=\inf\Big(\big\{\int_{\R^d}|\nabla u(\e)|^2\,
\mathrm{d}\e\,:\;u\in C^\infty(\R^d), \;u\ge
1\;\mathrm{on}\;E,\nonumber\\
&&u(\e)\rightarrow 0\;\mathrm{as}\;|\e|\rightarrow\infty \big\}\Big).
\end{eqnarray}
It is known (\cite{Ch}, \cite{Maz}, \cite{Maz1}) that the set function ``cap'' can be extended in a suitable manner from the collection of all compact subsets of the space $\R^d$ to the collection of all Borel subsets of it and the set function ``cap''
is monotonic and {\it submodular} (concave) in the sense that for any pair of sets $A,\,B\in\Sigma_B(\bar\Omega)$
\begin{equation}\label{submod}
\mathrm{cap}(A\cup B)+\mathrm{cap}(A\cap
B)\le\mathrm{cap}(A)+\mathrm{cap}(B)
\end{equation}
and the {\it isocapacity inequality} is valid:
\begin{equation}\label{isocapineq}
\forall\,F\in\Sigma_B(\bar\Omega):\quad\mathrm{mes}_d(F)\le
c_d\,(\mathrm{cap}(F))^{d/(d-2)}.
\end{equation}
with $c_d=\big(d(d-2)(\mathrm{mes}_d(B_1(0)))^{2/d}\big)^{-d/(d-2)}$, which  comes as
identity if $F$ is a closed ball.

In \cite{Zel1} we have used the following concepts:
\begin{definition}\label{dfregpar1}
	We call a subset of $\R^d$ a {\it regular parallelepiped}, if it has the form $\times_{k=1}^d[a_k,\,b_k]$.
\end{definition}
\begin{definition}\label{dfdenslogmtetpart1}
	Suppose that $m>1$	and $\theta\in(0,1)$. A sequence $\{D_n\}_{n=1}^\infty$ of subsets of a cube $Q_1(\by)$ is said to be a {\it $(\log_m,\,\theta)$- dense system} in $Q_1(\by)$, if 
	
	(a) each $D_n$ is a finite union of regular parallelepipeds;
	
	(b) for any 
	cube $Q_r(\bz)\subseteq Q_1(\by)$ with $r\in\big(0,\,\min\{1,\,\frac{1}{\theta m^2}\}\big)$ there is  
	$j\in\{1,2,\dots, \big[\log_m\big(\frac{1}{\theta r}\big)\big]\}$
	such that for some regular parallelepiped $\Pi\subseteq D_j$ there is a cube $Q_{\theta r}(\bs)$, contained in $\Pi\cap Q_r(\bz)$.
\end{definition}

\section{Some results from the previous work} \label{sec:prevwork}
\setcounter{equation}{0}

All the results of our previous paper \cite{Zel1} and partly of the present one are based on the following optimization problem for set functions:

\begin{problem}\label{extrprobl}
	Let $(X,\,\Sigma,\,\mu)$ be a measure space with a non-negative measure
	$\mu$  and $W(x)$ be a
	function defined on $X$ and belonging to
	$L_1(X,\,\mu)$. For $t\in(0,\,\mu(X))$ consider the
	collection $\mathcal{E}(t,X,\,\mu)$ of all $\mu$-measurable
	sets $E\subseteq X$ such that $\mu(E)\ge t$. The goal is to find
	the quantity
	\begin{equation}\label{dfIVOm}
	I_W(t,\,X,\,\mu)=\inf_{E\in\mathcal{E}(t,X,\,\mu)}\int_E W(x)\,
	\mu(\mathrm{d}x).
	\end{equation}
\end{problem}
\begin{remark}\label{rembinprob}
Problem \ref{extrprobl} can be formulated in the following equivalent form: to find
the quantity
\begin{equation*}
I_W(t,\,X,\,\mu)=\inf_{E\in\Sigma}\int_X W(x)\chi_E(x)\,
\mu(\mathrm{d}x)
\end{equation*}
under the constrain $\int_X\chi_E(x)\,
\mu(\mathrm{d}x\ge t$. We see that this problem is an infinite-dimensional generalization of a binary problem from integer linear programming \cite{N-W}.
\end{remark}
\begin{remark}\label{remonVge0}
In \cite{Zel1} we have assumed in the formulation of Problem \ref{extrprobl} that $W(\e)\ge 0$.  Here we omit this assumption, since the method of Lagrangian relaxation, used below for lower estimation of $I_W(t,\,X,\,\mu)$, does not need it. 
\end{remark}
In the formulation of next claim, proved in \cite{Zel1}, we have  used the following notations.
For the measure space and the function $W(x)$, introduced in Problem \ref{extrprobl},
consider the quantity:
\begin{equation}\label{dfJVOm}
J_W(t,\,X,\mu):=\int_{\mathcal{K}_W^-\big(t,\,X,\,\mu\big)}W(x)\,
\mu(\mathrm{d}x)+\big(t-\kappa_W^-(t,\,X,\,\mu)\big) W_\star(t,\,X,\,\mu),
\end{equation}
where $W_\star(t,\,X,\,\mu)$ is the non-decreasing rearrangement of the function $W(x)$, i.e,,
\begin{equation}\label{dfstOm}
W_\star(t,\,X,\,\mu):=\sup\{s>0:\;\lambda_\star(s,\,W,\,X,\,\mu)<
t\}\quad (t>0)
\end{equation}
with
\begin{equation}\label{dfFVOm}
\lambda_\star(s,\,W,\,X,\,\mu)=\mu({\mathcal L} _\star(s,W,X)),
\end{equation}
\begin{equation}\label{dfKVOm}
{\mathcal L} _\star(s,W,X)=\{x\in X:\;W(x)\le s\}.
\end{equation}
Furthermore,
\begin{equation}\label{dfKWt}
\mathcal{K}_W^-\big(t,\,X,\,\mu\big)={\mathcal L} _\star(s^-,W,X)\vert_{s=W_\star(t,\,X,\,\mu)},
\end{equation}
and
\begin{equation}\label{dfFVommin}
\kappa_W^-(t,\,X,\,\mu)=\mu\big(\mathcal{K}_W^-(t,\,X,\,\mu)\big),
\end{equation}
where
\begin{equation}\label{dfKVsminOm}
{\mathcal L} _\star(s^-,W,X)=\bigcup_{u<s}{\mathcal L} _\star(u,W,X)=\{x\in
X:\;W(x)< s\},
\end{equation}

The following claim from \cite{Zel1} solves Problem \ref{extrprobl} for a non-atomic  measure:

\begin{proposition}\label{prsolextrprob}[\cite{Zel1}, Theorem 3.3]
	Suppose that, in addition to  conditions of Problem \ref{extrprobl}, $W(\e)\ge 0$ and the measure 
	$\mu$ is non-atomic. Then
	
\noindent(i) for any $t\in(0,\,\mu(X))$ there exists a $\mu$-measurable set $\tilde{\mathcal K}\subseteq X$ such that 
	\begin{equation}\label{estmesKV}
	\mu(\tilde{\mathcal K})=t,
	\end{equation}
	for the quantity $J_W(t,\,X,\,\mu)$, defined by
	\eqref{dfJVOm}-\eqref{dfFVommin}, the representation
	\begin{equation}\label{reprJV}
	J_W\big(t,\,X,\,\mu\big)= \int_{\tilde{\mathcal
			K}}W(x)\, \mu(\mathrm{d}x)
	\end{equation}
	is valid and
	\begin{eqnarray}\label{propKV}
	&&\forall\;x\in\tilde{\mathcal K}:\quad W(x)\le
	W_\star(t,\,X,\,\mu)\nonumber,\\
	&&\forall\;x\in X\setminus\tilde{\mathcal K}:\quad W(x)\ge W_\star(t,\,X,\,\mu);
	\end{eqnarray}
	\vskip2mm
	\noindent(ii) the equality
	\begin{equation}\label{IeqJ}
	I_W(t,\,X,\,\mu)=J_W(t,\,X,\,\mu)
	\end{equation}
	is valid.
\end{proposition}

In the next claim. proved in \cite{Zel1}, we have obtained a two-sided estimate for the solution $J_W(t,\,X,\,\mu)$ of Problem \ref{extrprobl}  via a non-increasing rearrangement of the function $W(x)$ on  $X$. This rearrangement is following:
\begin{equation}\label{dfSVyrdel}
 W^\star(t;\,X;\,\mu):=\sup\{s>0\,:\;\lambda^\star(s;\,\,W;\,X;\,\mu)\ge
t\}\quad (t>0),
 \end{equation}
where
\begin{eqnarray}\label{dfLVsry}
&&\lambda^\star(s;\,W;\,X;\,\mu)=\mu(\mathcal{L}^\star(s;\,W;\,X)),\nonumber\\
&&\mathcal{L}^\star(s;\,W;\,X)=\{x\in X,:\; W(x)\ge s\}.
\end{eqnarray}

The promised claim is following:

\begin{proposition}\label{lmJR}[\cite{Zel1}, Proposition 3.4]
	 Suppose that, in addition to conditions of Problem \ref{extrprobl} and Proposition
	\ref{prsolextrprob}, the measure $\mu$ is finite.
	Then for $\theta>1$	and $t\in(0,\,\mu(X))$ the estimates
	\begin{equation}\label{estJR1}
	J_W(\mu(X)-t/\theta,\,X,\,\mu))\ge\frac{(\theta-1)t}{\theta} W^\star(t;\,X;\,\mu),
	\end{equation}
	\begin{equation}\label{estJR2}
	J_W(\mu(X)-t,\,X,\,\mu)\le (\mu(X)-t) W^\star(t;\,X;\,\mu)
	\end{equation}
	are valid.
\end{proposition}

\section{Main results} \label{sec:mainres}
\setcounter{equation}{0} 

\subsection{Non-increasing rearrangement}

In this subsection we shall assume that $V(\e)\ge 0$.

Denote by $\mathcal{M}_{\gamma}(\by,r)\;(\gamma\in(0,1))$ the
collection of all Borel sets $F\subseteq\mathcal{G}_r(\by)$
satisfying the condition
$\mathrm{mes}_d(F)\le\gamma\,\mathrm{mes}_d(\mathcal{G}_r(\by))$, where the domain $\mathcal{G}_r(\by)$ is defined by 
\eqref{dfcalGry}, \eqref{formofcalG}. A direct use of isocapacity inequality \eqref{isocapineq}
leads to the following claim:
\begin{theorem}\label{thcondlebesgue}
	If for some $r_0>0$ and a  function $\gamma(r)$,	satisfying the conditions 
	\begin{equation}\label{cndtildgam}
	\forall\,r\in(0,\,r_0):\;\gamma(r)\in(0,1)\quad\mathrm{and}\quad\limsup_{r\downarrow
		0}\,r^{-2d/(d-2)}\,\gamma(r)=\infty,
	\end{equation}
	the condition
	\begin{equation}\label{cndlebesgue}
	\lim_{|\by|\rightarrow\infty}\inf_{F\in\mathcal{M}_{\gamma(r)}(\by,r)}\int_{\mathcal{G}_r(\by)\setminus
		F}V(\e)\, \mathrm{d}\e=\infty
	\end{equation}
	is satisfied for any $r\in(0,r_0]$, then the spectrum of the operator
	$H=-\Delta+V(\e)$ is discrete.
\end{theorem}
\begin{remark}\label{remmistake}
	In \cite{Zel1} in condition \eqref{cndtildgam} the technical mistake was admitted:  the exponent $-2(d-2)/d$ of $r$  was there instead of $-2d/(d-2)$.
\end{remark}

Let us introduce some notations.
We shall omit $\mu$ in the  notation \eqref{dfSVyrdel} of non-increasing rearrangement in the case where $X$ is an open domain $\Omega\subseteq\R^d$ and $\mu=\mathrm{mes}_d$  
, i.e., to write  $ W^\star(t,\Omega)$. 
In the case where $\Omega=\mathcal{G}_r(\by)$  we shall write  $W^\star(t,\by,r)$.  
If $\mathcal{G}_r(\by)$ is a cube $Q_r(\by)$, we shall use the same notation, if it could not lead to a confusion.

On the base of Theorem \ref{thcondlebesgue} we get the following claim:
\begin{theorem}\label{thcondlebesrear}
	If for some $r_0>0$ and a function $\gamma(r)$, satisfying the conditions \eqref{cndtildgam}, 
the condition
	\begin{equation}\label{cndlebesrear}
	\lim_{|\by|\rightarrow\infty} V^\star(\delta(r),\by,r)=\infty
	\end{equation}
	is fulfilled  for any $r\in(0,r_0]$ with $\delta(r)=\gamma(r)\mathrm{mes}_d(\mathcal{G}_r(0)$, then the spectrum of the operator
	$H=-\Delta+V(\e)$ is discrete.
\end{theorem}
\begin{remark}\label{remequivcond}
	Notice that condition \eqref{cndlebesrear} is easier verifiable than condition \eqref{cndlebesgue}. On the other hand, Proposition \ref{prsolextrprob} and 
	estimates \eqref{estJR1}, \eqref{estJR2} imply that these conditions are equivalent in the following sense: for some function $\tilde\gamma(r)$	satisfying conditions \eqref{cndtildgam} the condition \eqref{cndlebesgue} is fulfilled  if and only if
	for some function $\gamma(r)$	satisfying conditions \eqref{cndtildgam} the condition
	\eqref{cndlebesrear} is fulfilled. 
\end{remark}

\begin{remark}\label{remGMD}
	From Theorem 6 of \cite{GMD} the following criterion of
	discreteness of the spectrum for operator $H=-\Delta+V(\e)$
	follows (in our notations): if for some numbers $\delta>0$,
	$c\in(0,1)$ and $r_0>0$ and for any $\by\in\R^d$, $r\in(0, r_0)$
	the condition
	\begin{equation}\label{cndGMD}
	\lambda^\star\Big(\frac{\delta}{\mathrm{mes}_d(Q_r(0))}\int_{Q_r(\by)}V(\e)\,d\e,V,\,r,\,\by\Big)\ge
	c\;\mathrm{mes}_d(Q_r(0))
	\end{equation}
	is fulfilled, then discreteness of the spectrum of the operator $H$ is
	equivalent to the condition:
	$\lim_{|\by|\rightarrow\infty}\int_{Q_r(\by)}V(\e)\,d\e=\infty$
	for some (hence for every) $r>0$. 
	 It is easy to see that as a sufficient condition  this
	criterion follows from Theorem \ref{thcondlebesrear} with $\mathcal{G}_r(\by)=Q_{r_1}(\by-r\vec a)$ 
	($r_1=2r/\sqrt{d}$, $\vec a=(d^{-1/2}, d^{-1/2},\dots,d^{-1/2})$), if one takes
	$\gamma(r)\equiv c$.  
\end{remark}

Consider the covering of the space $\R^d$ by the cubes $Q_1(\vec l)\;(\vec l\in\Z^d)$ and for any $\vec l\in\Z^d$ consider a sequence $\{D_j(\vec l)\}_{j=1}^\infty$ of 
subsets of $Q_1(\vec l)$. Furthermore, for some integers $n>0$ and $m>1$ consider the $m$-adic partition of each cube $Q_1(\vec l)$: $\{Q(\vec\xi,n)\}_{\vec\xi\in\Xi(\vec l,\,n)}$,
where $Q(\vec\xi,n)=Q_{m^{-n}}(\vec\xi)$ and $\Xi(\vec l,\,n)=\{\vec\xi\in m^{-n}\cdot\Z^d\,:\,Q(\vec\xi,n)\subset  Q_1(\vec l)\}$. Denote
\begin{equation}\label{dfXijvecl}
\Xi_j(\vec l,\,n)=\{\vec\xi\in m^{-n}\cdot\Z^d\,:\,Q(\vec\xi,n)\subseteq  D_j(\vec l)\},
\end{equation} 

The following claim is based on the previous claim and on the concept of a $(\log_m,\,\theta)$- dense system (Definition \ref{dfdenslogmtetpart1}):

\begin{theorem}\label{prlogmdensdoubrearr}[\cite{Zel1}, Theorem 3.14]
	Suppose that $\theta\in(0,1)$ and for each $\vec l\in\Z^d$ the sequence $\{D_j(\vec l)\}_{j=1}^\infty$ forms a $(\log_m,\,\theta)$- dense system in $Q_1(\vec l)$. Furthermore,
	suppose that 
	\begin{equation}\label{Xinonempty2}
	\forall\;\vec l\in\Z^d,\;j\in\N,\;:\quad \Xi_{j}(\vec l,j)\neq\emptyset. 
	\end{equation}
 Let $\gamma(r)$ be a non-decreasing monotone function 
	satisfying condition \eqref{cndtildgam}. If for 
	any natural $n$ the condition 
	\begin{equation}\label{cndlogdensdoubrearrprev}
	\lim_{|\vec l|\rightarrow\infty}\;\min_{\vec\xi\in\bigcup_{j=1}^n\Xi_{j}(\vec l,\,n)} V^\star\big(\psi(n);\,Q(\vec\xi,n)\big)=\infty
	\end{equation}
	is satisfied with $\psi(n)=\gamma(m^{-n})\mathrm{mes}_d(Q(\vec 0,n))$, then the spectrum of the operator $H=-\Delta+V(\e)\cdot$ is discrete.
\end{theorem}

\subsection{Lagrangian relaxation}

Let us return to Problem \ref{extrprobl} and assume that  $(X,\,\Sigma,\,\mu)$ is a probability space, i.e. $\mu(X)=1$.  We shall estimate from below the value $I_W(t,\,X,\,\mu)$, defined by \eqref{dfIVOm}, using the method of Lagrangian relaxation (\cite{YRMR}). Before formulation of the results let us introduce some notations. Consider the expectation $\mathrm{E}(W)$ of the function $W(x)$: 
\begin{equation}\label{dfedxpec}
\mathrm{E}(W)=\int_XW(x)\mu(dx)
\end{equation}
 and its (standard) deviation 
\begin{eqnarray}\label{dfVar}
&&\mathrm{Dev}(W)=\sqrt{E\big((W-\mathrm{E}(W))^2\big)}=\sqrt{\int_X\big(W(x)-\mathrm{E}(W)\big)^2\mu(dx)}=\nonumber\\
&&\sqrt{\mathrm{E}(W^2)-\big(\mathrm{E}(W)\big)^2}, 
\end{eqnarray}
assuming that  $W\in L_2(X,\mu)$.
The following claim is valid:
\begin{theorem}\label{thestLagr}
	If $t\in(1/2,1)$ and $\sigma(t)=2t-1$, then
\begin{equation}\label{estLagr}
I_W(t,\,X,\,\mu)\ge(1/2)(1+\sigma(t))\mathrm{E}(W)-(1/2)\sqrt{1-\sigma^2(t)}\cdot\mathrm{Dev}(W)
\end{equation}
\end{theorem}

Consider the probability space $\big(\mathcal{G}_r(\by),\Sigma_L(\mathcal{G}_r(\by)),m_{d,r}\big)$,
where $m_{d,r}$ is the normalized Lebesgue measure
\begin{equation}\label{dfmdr}
m_{d,r}(A):=\frac{\mathrm{mes}_d(A)}{\mathrm{mes}_d({\mathcal G}_r(\by))}\quad(A\in\Sigma_L(\bar {\mathcal G}_r(\by))).
\end{equation}
 Denote by $\mathrm{E}_{\by,r}(W)$ and $\mathrm{Dev}_{\by,r}(W)$ the expectation and deviation of a real-valued  function $W\in L_2(\mathcal{G}_r(\by),\, m_{d,r})$. 
 
The following claim is based on Theorem \ref{thestLagr}:
\begin{theorem}\label{thdiscLagr}
Suppose that  $V(\e)\ge 0$,
$V\in L_{2,loc}(\R^d)$ and the condition 
\begin{equation}\label{cnddisclarg}
\lim_{|\by|\rightarrow\infty}\Big(\mathrm{E}_{\by,r}(V)-\sqrt{\gamma(r)}\cdot\mathrm{Dev}_{\by,r}(V)\Big) =+\infty
\end{equation}
is satisfied for some $r_0>0$ and any $r\in(0,r_0)$, where $\gamma(r)$ satisfies the condition \eqref{cndtildgam}. Then the spectrum of the operator
$H=-\Delta+V(\e)$ is discrete.
\end{theorem}

\begin{remark}\label{remLagrimplrear}
We shall see from the proof of Theorem \ref{thdiscLagr} that condition \eqref{cnddisclarg} implies condition \eqref{cndlebesgue} of Theorem \ref{thcondlebesgue}. Hence, in view of Remark \ref{remequivcond}, conditions of Theorem \ref{thdiscLagr} imply conditions of  Theorem \ref{thcondlebesrear}.
\end{remark}

\subsection{Perturbations of the potential}
\label{subsec:perturb}

In this section we shall use  a localization
principle, which was established in \cite{Iwas}  and \cite{K-Sh}  for the case
of the magnetic Schr\"odinger operator, but in particular it is true
also in absence of the magnetic field.
Let $\Omega$ be an open domain in $\R^d$ whose closure is compact.
 We denote by $(\cdot,\cdot)_\Omega$ and $\Vert\cdot\Vert_\Omega$ the
inner product and the norm in the space $L_2(\Omega)$. Consider
the quantities:
\begin{eqnarray*}
&&\lambda_0(\Omega):=\inf_{u\in C_{0}^\infty(\Omega),\,u\neq
	0}\frac{(H u,u)_\Omega}{\Vert u\Vert_\Omega^2}=\nonumber\\
&&\inf_{u\in C_{0}^\infty(\Omega),\,u\neq 0}\frac{\int_\Omega(|\nabla
	u(\e)|^2+V(\e)|u(\e)|^2)\, \mathrm{d}\e}{\int_\Omega|u(\e)|^2\,
	\mathrm{d}\e},
\end{eqnarray*}
\begin{equation*}
\mu_0(\Omega):=\inf_{u\in C^\infty(\Omega),\,u\neq
	0}\frac{(H u,u)_\Omega}{\Vert u\Vert_\Omega^2}
\end{equation*}
In essence $\lambda_0(\Omega)$ is the minimal eigenvalue of the
generalized Dirichlet  boundary problem in the domain $\Omega$:
$Hu=\lambda u\;u\vert_{\partial\Omega}=0$ and $\mu_0(\Omega)$ is the minimal eigenvalue of the generalized Neumann problem for the same equation, i.e., the corresponding boundary condition is $\frac{\partial u}{\partial\vec n}\vert_{\partial\Omega}=0$.  As it is known, the spectra
of these problems are discrete. We denote briefly
$\lambda_0(\by,r)=\lambda_0(Q_r(\by))$ and $\mu_0(\by,r)=\mu_0(Q_r(\by))$
\begin{proposition}\label{prloc}[\cite{K-Sh} Theorems 1.1, 1.2]
	
	\noindent$(i)$ The operator $H$ is bounded below if and only if
	\begin{equation}\label{boundbel}
	\exists\;r>0:\quad
	\liminf_{|\by|\rightarrow\infty}\lambda_0(\by,r)>-\infty;
	\end{equation}
	\vskip2mm
	
	\noindent$(ii)$ The spectrum of $H$ is discrete and bounded below if
	and only if
	\begin{equation*}
	\exists\;r>0:\quad
	\lim_{|\by|\rightarrow\infty}\lambda_0(\by,r)=+\infty;
	\end{equation*}
	\vskip2mm
	
	\noindent$(iii)$ If $V(\e)\ge 0$, the spectrum of $H$ is discrete if
	and only if
	\begin{equation*}
	\exists\;r>0:\quad
	\lim_{|\by|\rightarrow\infty}\mu_0(\by,r)=+\infty.
	\end{equation*}
\end{proposition}

Let $\Omega$ be an open and bounded domain in $\R^d.$ Following to \cite{M-V}, for $W\in L_1(\Omega)$ and $s\ge 1$ consider the
 set ${\mathcal D}_s(W,\,\Omega)$ of all vector fields $\vec\Gamma(\e)$ on $\Omega$ satisfying the divergence equation $\mathrm{div}\vec\Gamma=W$ and belonging to the space $L_s(\Omega)$.
We assume that  this equation  is
satisfied in the generalized sense (see \eqref{defgensol}). By Propositions 
\ref{prdiv} and \eqref{prdiveq} ${\mathcal D}_s(W,\,\Omega)\neq\emptyset$ under some  additional conditions.
Denote
\begin{equation}\label{defDsW}
D_s(W,\Omega)=\inf_{\vec\Gamma\in{\mathcal D}_s(W,\,\Omega)}\Vert\vec\Gamma\Vert_s,
\end{equation}
In the case where $\Omega=Q_r(\by)$ we shall write briefly 
$D_s(W,\by,r)$.
Consider the constant:
\begin{equation}\label{dfC}
C(d)=\sqrt{\frac{1}{\pi
		d(d-2)}}\Big(\frac{\Gamma(d)}{\Gamma(d/2)}\Big)^{1/d},
\end{equation}

The following claim on perturbation of Schr\"odinger operator, preserving discreteness of its spectrum, is valid:

\begin{theorem}\label{thperturb}
Suppose that the potential $V(\e)$ 
 permit the representation
$V(\e)=V_0(\e)+W(\e)$ such that $V_0\in L_{1,loc}(\R^d)$
and for some $r_0>0$ the condition 
\begin{equation}\label{cndpert}
\bar
D(r_0,W)=\limsup_{|\by|\rightarrow\infty}D_d(W,\by,r_0)<
1/(2C(d)),
\end{equation}
is satisfied.

\noindent (i) If for some 
\begin{equation}\label{sigmain}
\sigma\in[]2C(d)\bar D(r_0,W),\,1)
\end{equation}
the operator $H_\sigma=-(1-\sigma)\Delta+V_0(\e)\cdot$ is bounded below,
then so does operator $H=-\Delta+V(\e)\cdot$ ;
\vskip2mm

\noindent (ii)  If additionally the spectrum of operator $H_\sigma$ is
discrete, then so does the spectrum of operator $H$.
\end{theorem}

In the next claim we shall use the Fourier transform $\mathcal{F}(f)$  on $\T^d$ of a function $f\in L_2(\T^d)$.

 \begin{theorem}\label{thFour}
Suppose that  $V\in L_{2,\,loc}(\R^d)$ and 
\begin{equation}\label{cndtransfour}
\limsup_{|\vec l|\rightarrow\infty}\Big\Vert \frac{\mathcal{F}\big(W_{\vec l}\big)(\vec k)}{2\pi i|\vec k|^2}\vec k\Big\Vert_{l_{d^\prime}}<\frac{1}{d2^{d+1}C(d)}\quad (\vec l\in\Z^d),
\end{equation}
where $d^\prime=d/(d-1)$ and $W_{\vec l}(\e)\,(\vec l\in\Z^d)$ is the $1$-periodic continuation (by all the variables) of the function $\big(V(\e)-\int_{Q_1(\vec l)}V(\bs)\,\mathrm{d}\bs\big)\vert_{Q_1(\vec l)}$ from the cube $Q_1(\vec l)$ to the whole $\R^d$. Then the spectrum of the operator $H$ is discrete and bounded below, if
\begin{equation}\label{intQitendinfty}
\lim_{|\vec l|\rightarrow\infty}\int_{Q_1(\vec l)}V(\bs)\,\mathrm{d}\bs=+\infty.
\end{equation}
\end{theorem}

The next claim is motivated by a counterexample, constructed by A.M. Molchanov in \cite{Mol}, and it answers the question, raised in Introduction. Let us construct the potential $V(\e)$ in the following manner. Consider a function 
\begin{equation}\label{SinL2}
S\in L_{2}(\R^d)
\end{equation}
 such that 
\begin{equation}\label{intSposit}
S(\e)\ge 0, \quad \mathrm{supp}(S)\subseteq \bar B_{1/2}(\vec 0)\quad\mathrm{and}\quad\int_{\R^d}S(\e)\,\mathrm{d}\e>0,
\end{equation} 
 its dilation $S(\e/r)\,(r\in(0,1/2))$, translation 
 \begin{equation*}
S\big((\e-\vec c)/r\big)\quad(\vec c=(1/2,\,1/2,\,\dots,\,1/2)) 
 \end{equation*}
 and the periodic summation of the last function
\begin{equation}\label{persumS}
V_r(\e)=\sum_{\vec k\in\Z^d}S\big((\e-\vec c-\vec k)/r\big),
\end{equation} 
which is $1$-periodic by all  the variables $x_1,\,x_2,\,\dots,\,x_d$.
For a natural $m$ consider dilation $V_r(m\e)$ of $V_r(\e)$, which is $1/m$-periodic by all  the variables. For some  functions $\mathcal{N}:\,\Z^d\rightarrow\R_+$, $r:\,\Z^d\rightarrow(0,\,1/2)$ and $m:\,\Z^d\rightarrow\N$ define the desired potential $V(\e)$ in the following way:
\begin{equation}\label{dfVrbysumm}
V(\e):=\mathcal{N}(\vec l)V_{r(\vec l)}\big(m(\vec l)\e\big)\quad\mathrm{for}\quad\e\in Q_1(\vec l).
\end{equation}
We shall use Fourier transform $\hat S(\vec\omega)$ on $\R^d$ of the function $S(\e)$.
Denote by $\chi_r(\e)$ the characteristic function of the ball $\bar B_r(\vec 0)$. Recall that the constant $C(d)$ is defined by \eqref{dfC}

The promised claim is following::

\begin{theorem}\label{thlattsupp}
Suppose that along with \eqref{SinL2}, \eqref{intSposit} the condition 
\begin{equation}\label{inclHr}
\bar H=\sup_{r\in(0,\,1)}r^d\Vert  H_r\Vert_{l_{d^\prime}(\Z^d)}<\infty 
\end{equation}	
is fulfilled with $d^\prime=d/(d-1)$, where
\begin{eqnarray}\label{dfHrveck}
H_r(\vec k)=\left\{\begin{array}{ll}
\frac{\hat S(r\vec k)}{2\pi |\vec k|}&\mathrm{for}\quad \vec k\in\Z^d\setminus\{\vec 0\},\\
0&\mathrm{for}\quad \vec k=\vec 0.
\end{array}
\right.
\end{eqnarray}
Then

\noindent (i) the conditions
\begin{equation}\label{cndfracNlml}
\limsup_{|\vec l|\rightarrow\infty}\frac{\mathcal{N}(\vec l)}{m(\vec l)}<\big(d2^{d+1}C(d)\bar H\big)^{-1}
\end{equation}
and 
\begin{equation}\label{cndNtimesrd}
\lim_{|\vec l|\rightarrow\infty}\mathcal{N}(\vec l)(r(\vec l))^d=\infty
\end{equation}
imply discreteness of the spectrum of $H=-\Delta+V(\e)\cdot$;
\vskip2mm

\noindent (ii) if $S(\e)=\chi_{1/2}(\e)$, condition \eqref{inclHr}-\eqref{dfHrveck} is satisfied for $d\in\{3,\,4\}$.
\end{theorem}
\begin{remark}\label{remonint}
	In view of \eqref{intSposit}-\eqref{dfVrbysumm}, we have for $\vec l\in\Z^d:$
\begin{eqnarray*}
&&\int_{Q_1(\vec l)}V(\e)\,\mathrm{d}\e=\mathcal{N}(\vec l)\int_{Q_1(\vec 0)}V_{r(\vec l)}(m(\vec l)\e)\,\mathrm{d}\e=\\
&&\frac{\mathcal{N}(\vec l)}{(m(\vec l))^d}\int_ {Q_{m(\vec l)}(\vec 0)}V_{r(\vec l)}(\bs)\,\mathrm{d}\bs
=\mathcal{N}(\vec l)\int_{Q_1(\vec 0)}V_{r(\vec l)}(\e)\,\mathrm{d}\e=\\
&&\mathcal{N}(\vec l)\int_{\R^d}S\big((\e-\vec c)/r(\vec l)\big)\,\mathrm{d}\e=\mathcal{N}(\vec l)\big(r(\vec l)\big)^d\int_{\R^d}S(\e)\,\mathrm{d}\e.
\end{eqnarray*}
Hence condition \eqref{cndNtimesrd} is equivalent to \eqref{intQitendinfty}.
\end{remark}

\begin{remark}\label{remoncompar}
If $S(\e)=\chi_{1/2}(\e)$, then using arguments from the previous remark, we get:
\begin{eqnarray*}
&&\mathrm{mes}_d\big(\{\e\in Q_1(\vec l):\;V(\e)>0 \} \big)=\int_{Q_1(\vec 0)}V_{r(\vec l)}(m(\vec l)\e)\,\mathrm{d}\e=\\
&&
\big(r(\vec l)\big)^d\mathrm{mes}_d(B_{1/2}(\vec 0)).
\end{eqnarray*}
Hence if  $\lim_{|\vec l|\rightarrow\infty}r(\vec l)=0$, the potential $V(\e)$, constructed above, does not satisfy condition \eqref{cndlebesrear} of Theorem \ref{thcondlebesrear} and condition \eqref{cnddisclarg} of Theorem \ref{thdiscLagr}. But in this case we can choose the functions $\mathcal{N}(\vec l)$ and $m(\vec l)$ such that they satisfy conditions \eqref{cndfracNlml} and \eqref{cndNtimesrd} of Theorem \ref{thlattsupp}, i.e., the spectrum of $H=-\Delta+V(\e)\cdot$ will be discrete. 
\end{remark}

\begin{remark}\label{remcountexMolch}
	The idea of a counterexamle from \cite{Mol}, mentioned above, is following. The potential $V(\e)$ have been constructed as above with $S(\e)=\chi_{1/2}(\e)$,  such that 
\begin{equation}\label{cndforrl}
\forall\,\vec l\in\Z^d:\quad r(\vec l)\in(0,\,1/2)
\end{equation}
and condition \eqref{intQitendinfty} (i,e,. \eqref{cndNtimesrd}) is fulfilled. We see from \eqref{intSposit}-\eqref{dfVrbysumm} that on each cube $Q_1(\vec l)\;(\vec l\in\Z^d)$ the set $A(\vec l):=\mathrm{supp}(V)\cap Q_1(\vec l)$ is a union of $\big(m(\vec l)\big)^d$ balls having the radius $a(\vec l)=r(\vec l)/m(\vec l)$. Consider the set $B(\vec l)$, which is the union of all balls concentric with these balls and having the radius $b(\vec l)=2a(\vec l)$. In \cite{Mol} on each cube $Q_1(\vec l)$ a nonnegative ``test function'' $\phi_{\vec l}(\e)$ have been constructed such that it belongs to $W_2^1(Q_1(\vec l))$ and has the properties: $\phi_{\vec l}(\e)=0$ for $\e\in A(\vec l)$, $\phi_{\vec l}(\e)=1$ for $\e\in Q_1(\vec l)\setminus B(\vec l)$ and it is defined in $B(\vec l)\setminus A(\vec l)$ such that
\begin{eqnarray}\label{cndsmallcap}
&&\int_{Q_1(\vec l)}|\nabla\phi_{\vec l}(\e)|^2\,\mathrm{d}\e=(d-2)\omega_d\big(m(\vec l)\big)^d
\frac{\big(a(\vec l)\big)^{d-2}}{1-\big(a(\vec l)/b(\vec l)\big)^{d-2}}=\nonumber\\
&&G_d\big(m(\vec l)\big)^2r(\vec l)\big)^{d-2}<1,
\end{eqnarray}
where $\omega_d$ is the volume of unit sphere in $\R^d$ and $G_d=(d-2)\omega_d\frac{2^{d-2}}{2^{d-2}-1}$. Condition \eqref{cndsmallcap} was called in \cite{Mol} ``small capacity property'' of the sets $A(\vec l)$. On the other hand, it is easy to show using \eqref{cndforrl} that for some $c>0$
\begin{equation*}
\forall\,\vec l\in\Z^d:\quad \int_{Q_1(\vec l)}|\phi_{\vec l}(\e)|^2\,\mathrm{d}\e\ge\mathrm{mes}_d\big(Q_1(\vec l)\setminus B(\vec l)\big)\ge c.
\end{equation*}
Then, taking into account the density of $C^\infty(Q_1(\vec l))$ in $W_2^1(Q_1(\vec l))$, we obtain:
\begin{eqnarray*}
&&\mu_0(Q_1(\vec l))=\inf_{u\in W_2^1(Q_1(\vec l)) ,\,u\neq 0}\frac{\int_{Q_1(\vec l)}(|\nabla
	u(\e)|^2+V(\e)|u(\e)|^2)\,\mathrm{d}\e}{\int_{Q_1(\vec l)}|u(\e)|^2\,\mathrm{d}\e}\le\\
&&\frac{\int_{Q_1(\vec l)}|\nabla\phi_{\vec l}(\e)|^2\,\mathrm{d}\e}{\int_{Q_1(\vec l)}|\phi_{\vec l}(\e)|^2\,\mathrm{d}\e}<1/c
\end{eqnarray*}
for all $\vec l\in\Z^d$. Hence by the localization principle (Proposition \ref{prloc}), the spectrum of operator $H=-\Delta+V(\e)\cdot$ is not discrete, despite condition \eqref{intVtendinf} is satisfied for the potential $V(\e)$. Notice that in view of \eqref{cndsmallcap} and \eqref{cndforrl}
\begin{equation*}
\frac{\mathcal{N}(\vec l)}{m(\vec l)}\ge 2^{d/2+1}\sqrt{G_d}\mathcal{N}(\vec l)(r(\vec l))^d. 
\end{equation*}
Hence \eqref{cndNtimesrd} implies that condition \eqref{cndfracNlml} of Theorem \ref{thlattsupp}
 is not satisfied. This fact is predictable, because otherwise by this theorem the spectrum of $H$ would be discrete for $d\in\{3,\,4\}$.
\end{remark}

\subsection{The case where the operator $H$ is defined in $L_2(\Omega)\;(\Omega\subset\R^d)$}\label{subsec:Omega}

Let us notice that in \cite{M-Sh}  Theorem \ref{thMazSh} have been formulated in a more general form. Assume that a nonnegative potential $V(\e)$ is defined in a open domain $\Omega\subset\R^d$ (which can be unbounded), $V\in L_{1,\,loc}(\Omega)$ and the Schr\"odinger operator $H=-\Delta+V(\cdot)$ acts in the space $L_2(\Omega)$. Denote by
$\mathcal{N}_{\gamma}(\by,r,\Omega)\;(\gamma\in(0,1))$ the collection of all
compact sets $F\subseteq\bar{\mathcal{G}}_r(\by)$ satisfying the conditions
\begin{equation}\label{inclGminOm}
\bar{\mathcal{G}}_r(\by)\setminus\Omega\subseteq F
\end{equation}
and \eqref{defNcap}, where $``\mathrm{cap}''$ is the harmonic capacity and the domain $\mathcal{G}_r(\by)$ is defined by 
\eqref{dfcalGry}, \eqref{formofcalG}. The more general formulation of Theorem \ref{thMazSh} is:
\begin{theorem}\label{thMazShOm}[\cite{M-Sh}, Theorem 2.2]
	The spectrum of the operator $H$ is discrete, if  
	for some $r_0>0$ and for any $r\in(0,r_0)$ the condition
	\begin{equation}\label{cndmolch1Om}
	\lim_{|\by|\rightarrow\infty}\inf_{F\in\mathcal{N}_{\gamma(r)}(\by,r,\Omega)}\int_{{\mathcal G}_r(\by)\setminus
		F}V(\e)\, \mathrm{d}\e=+\infty,
	\end{equation}
	is satisfied, where $\gamma(r)$ satisfies the conditions \eqref{cndgammar}.
	\end{theorem}
Denote by
\begin{equation*}
\mathcal{P}_{\gamma}(\by,r,\Omega)\quad(\gamma\in(0,1))
\end{equation*}
the
collection of all $\mathrm{mes}_d$-measurable sets $E\subseteq\mathcal{G}_r(\by)\cap\Omega$
satisfying the condition $\mathrm{mes}_d(E)\ge(1-\gamma)\mathrm{mes}_d(\mathcal{G}_r(\by))$.  For $W\in L_{1,\,loc}(\Omega)$ and $\gamma\in(0,1)$ consider the quantities:
\begin{eqnarray}\label{dfRWyrOm}
R_W(\by,r,\Omega,\gamma)=\left\{\begin{array}{ll}
+\infty,&\mathrm{if}\\\mathrm{mes}_d(\mathcal{G}_r(\by)\cap\Omega)<(1-\gamma)\,\mathrm{mes}_d(\mathcal{G}_r(\by)),\\
\\
\int_{\mathcal{G}_r(\by)\cap\Omega} W(\e)\,\mathrm{d}\e,&\mathrm{if}\\\mathrm{mes}_d(\mathcal{G}_r(\by)\cap\Omega)=\\(1-\gamma)\,\mathrm{mes}_d(\mathcal{G}_r(\by)),\\
\\
\inf_{E\in\mathcal{P}_{\gamma}(\by,r,\Omega)}\int_E W(\e)\,\mathrm{d}\e,&\mathrm{if}\\\mathrm{mes}_d(\mathcal{G}_r(\by)\cap\Omega)>(1-\gamma)\,\mathrm{mes}_d(\mathcal{G}_r(\by)),
\end{array}
\right.	
\end{eqnarray}

A direct use of isocapacity inequality \eqref{isocapineq} leads to the following claim:
\begin{theorem}\label{thmeasinstcapRVOm}
	Suppose that for some $r_0>0$ and any $r\in(0,r_0)$ the condition
	\begin{equation}\label{cndmeasinstcapcalGOm}
	\lim_{|\by|\rightarrow\infty}R_V(\by,r,\Omega,\gamma(r))=+\infty
	\end{equation}
	is fulfilled with a function $\gamma(r)$ satisfying condition \eqref{cndtildgam}.
	Then the spectrum of the operator $H=-\Delta+V(\e)$ is discrete.
\end{theorem}

Taking $V(\e)\equiv 0$, we obtain the following consequence of Theorem \ref{thmeasinstcapRVOm}:
\begin{corollary}\label{cordiscspecLaplace}
	If for some $r_0>0$ and any $r\in(0,r_0)$ the condition
	\begin{equation}\label{cnddiscspecLap}
	\limsup_{|\by|\rightarrow\infty}\frac{\mathrm{mes}_d(\mathcal{G}_r(\by)\cap\Omega)}{\mathrm{mes}_d(\mathcal{G}_r(\by))}\le 1-\gamma(r) 
	\end{equation}
	is fulfilled with a function $\gamma(r)$ satisfying condition \eqref{cndtildgam}, then the spectrum of the operator $-\Delta$ is discrete.
\end{corollary}

\begin{remark}\label{remcor48}
	Condition \eqref{cnddiscspecLap} can be written in the equivalent form:
	\begin{equation*}
	\liminf_{|\by|\rightarrow\infty}\frac{\mathrm{mes}_d(\mathcal{G}_r(\by)\setminus\Omega)}{\mathrm{mes}_d(\mathcal{G}_r(\by))}\ge \gamma(r) 
	\end{equation*}
	In this formulation for the case $\gamma(r)\equiv const$ the result of Corollary \ref{cordiscspecLaplace} was obtained in \cite{Kon-Shub} (Corollary 6.8).	
\end{remark}

Consider the probability space
\begin{equation*}
\big(\mathcal{G}_r(\by)\cap\Omega,\,\Sigma_L(\mathcal{G}_r(\by)\cap\Omega),\,m_{d,r,\Omega}\big),
\end{equation*}
 where $m_{d,r,\Omega}$ is the normalized Lebesgue measure on $\mathcal{G}_r(\by))\cap\Omega$:
 \begin{equation*}
 m_{d,r,\Omega}(A):=\frac{\mathrm{mes}_d(A)}{\mathrm{mes}_d({\mathcal G}_r(\by)\cap\Omega)}\quad(A\in\Sigma_L(\bar {\mathcal G}_r(\by)\cap\Omega)).
 \end{equation*}
  Denote by $\mathrm{E}_{\by,r,\Omega}(W)$ and $\mathrm{Dev}_{\by,r,\Omega}(W)$ the expectation and deviation of a function $W\in L_2(\mathcal{G}_r(\by)\cap\Omega,\,m_{d,r,\Omega})$. Consider the quantity:
 \begin{eqnarray}\label{dfYWyrOm}
 Y_W(\by,r,\Omega,\gamma)=\left\{\begin{array}{ll}
 +\infty,&\mathrm{if}\\\mathrm{mes}_d(\mathcal{G}_r(\by)\cap\Omega)<(1-\gamma/2)\,\mathrm{mes}_d(\mathcal{G}_r(\by)),\\
 \\
\mathrm{E}_{\by,r,\Omega}(W)-\sqrt{2\bar\gamma(\gamma,\Omega)}\cdot\mathrm{Dev}_{\by,r,\Omega}(W), &\mathrm{if}\\\mathrm{mes}_d(\mathcal{G}_r(\by)\cap\Omega)\ge(1-\gamma/2)\,\mathrm{mes}_d(\mathcal{G}_r(\by)),
 \end{array}
 \right.	
 \end{eqnarray} 
 where 
\begin{equation}\label{bargamOm}
\bar\gamma(\gamma,\Omega)=1-(1-\gamma/2)\frac{\mathrm{mes}_d(\mathcal{G}_r(\by))}{\mathrm{mes}_d({\mathcal G}_r(\by)\cap\Omega)}.
\end{equation} 

The following claim is based on Theorems \ref{thestLagr} and \ref{thmeasinstcapRVOm}:
\begin{theorem}\label{thdiscLagrOm}
	Suppose that $V\in L_{2,loc}(\Omega)$ and the condition 
	\begin{equation}\label{cnddisclargOm}
	\lim_{|\by|\rightarrow\infty} Y_V(\by,r,\Omega,\gamma(r))=+\infty
	\end{equation}
	is satisfied for some $r_0>0$ and any $r\in(0,r_0)$, where $\gamma(r)$ satisfies the condition \eqref{cndtildgam}. Then the spectrum of the operator
	$H=-\Delta+V(\e)$ is discrete.
\end{theorem}
\begin{remark}\label{reminfldev}
We see from \eqref{dfYWyrOm} and \eqref{bargamOm} that if in Theorem \ref{thdiscLagrOm} the quantity $\frac{(1-\gamma/2)\mathrm{mes}_d(\mathcal{G}_r(\by))}{\mathrm{mes}_d({\mathcal G}_r(\by)\cap\Omega)}$ approaches $1$ from below for large $|\by|$, the condition \eqref{cnddisclargOm} becomes weaker because of a weaker influence of  deviation of the potential $V(\e)$ in ${\mathcal G}_r(\by)\cap\Omega$.  On the other hand, if it is bigger than $1$, then by Corollary \ref{cordiscspecLaplace} even $V(\e)\equiv 0$ satisfies this condition.
\end{remark} 
 
\section{Proof of main results} \label{sec:proofmainres}
\setcounter{equation}{0}

\subsection{Proof of Theorem \ref{thcondlebesgue}}
\begin{proof}
Consider the following quantities connected 
with the domain $\mathcal{G}$ having the form \eqref{formofcalG}:
\begin{equation}\label{dfbarrrm}
\bar r:=\max_{\omega\in S^{d-1}}r(\omega),\quad r_m:=\min_{\omega\in S^{d-1}}r(\omega),
\end{equation}
\begin{equation}\label{dfconstG}
G:=\Big(\frac{\bar r}{r_m}\Big)^d.
\end{equation}
we see from \eqref{formofcalG}, \eqref{dfcalGry}, \eqref{dfbarrrm}
and \eqref{dfconstG} that $\mathcal{G}_r(\by)\subseteq B_{\bar r\cdot r}(\by)$ 
and 
\begin{equation*}
\mathrm{mes}_d(B_{\bar r\cdot r}(\by))\le G\cdot\mathrm{mes}_d(\mathcal{G}_r(\by)).
\end{equation*}
Let us define
\begin{equation}\label{dfgamcalG}
\tilde\gamma(r)=\big(\gamma(r)/G\big)^{(d-2)/d}.
\end{equation}
In view of \eqref{cndtildgam}. the function $\tilde\gamma(r)$ satisfies condition \eqref{cndgammar}.
Suppose that $F\in\mathcal{N}_{\tilde\gamma(r)}(\by,r)$, where the collection $\mathcal{N}_{\gamma}(\by,r)$
is defined by \eqref{defNcap}.  In view of
the isocapacity inequality \eqref{isocapineq}
and the fact that it comes as
identity for $F=\bar B_{\bar r\cdot r}(\by)$, we have, taking into account
\eqref{dfgamcalG}:
\begin{eqnarray*}\
	&&\frac{\mathrm{mes}_d(F)}{\mathrm{mes}_d
		(\mathcal{G}_r(\by))}\le G\frac{\mathrm{mes}_d(F)}{\mathrm{mes}_d
		(B_{\bar r\cdot r}(\by))}\le G\Big(\frac{\mathrm{cap}(F)}{\mathrm{cap}(\bar
		B_{\bar r\cdot r}(\by))}\Big)^{d/(d-2)}\le\nonumber\\
	&&  G\Big(\frac{\mathrm{cap}(F)}{\mathrm{cap}(\bar
		{\mathcal G}_r(\by))}\Big)^{d/(d-2)}\le G(\tilde\gamma(r))^{d/(d-2)}=\gamma(r),
\end{eqnarray*}
i.e., $F\in\mathcal{M}_{\gamma(r)}(\by,r)$. Thus,
$\mathcal{N}_{\tilde\gamma(r)}(\by,r)\subseteq\mathcal{M}_{\gamma(r)}(\by,r)$.
Hence
\begin{equation*}
\inf_{F\in\mathcal{N}_{\tilde\gamma(r)}(\by,r)}\int_{\mathcal{G}_r(\by)\setminus
	F}V(\e)\, \mathrm{d}\e\ge
\inf_{F\in\mathcal{M}_{\gamma(r)}(\by,r)}\int_{\mathcal{G}_r(\by)\setminus
	F}V(\e)\, \mathrm{d}\e.
\end{equation*}
Therefore in view of condition \eqref{cndlebesgue}, condition \eqref{cndmolch1} is satisfied with $\tilde\gamma(r)$ insyead of $\gamma(r)$. 
Hence by Theorem \ref{thMazSh} the spectrum of the operator $H$ is discrete.
Theorem \ref{thcondlebesgue} is proven.	
\end{proof}

\subsection{Proof of Theorem \ref{thcondlebesrear}}
\begin{proof}
	Let us take $\theta>1$, $\tilde\gamma(r)=\gamma(r)/\theta$ and $\sigma(r)=(1-\tilde\gamma(r))\mathrm{mes}_d(\mathcal{G}_r(0))$.
	Since  $\gamma(r)$ satisfies conditions \eqref{cndtildgam}, then $\tilde\gamma(r)$ satisfies these conditions. 
	Denote briefly for $t\in(0,1)$: $\mathcal{E}(t,\,\by\,,r):=\mathcal{E}(t,\,\mathcal{G}_r(\by),\,\mathrm{mes}_d)$
(Problem \ref{extrprobl}).
	 In view of Proposition \ref{prsolextrprob} and  estimate \eqref{estJR1} (Proposition \ref{lmJR}) with $W(\e)=V(\e)$, $X=\mathcal{G}_r(\by)$ and
	$t=\delta(r)$, we have:
	\begin{eqnarray*}\label{estJvfrbelow}
&&	\hskip-5mm\inf_{F\in\mathcal{M}_{\tilde\gamma(r)}(\by,r)}\int_{\mathcal{G}_r(\by)\setminus
		F}V(\e)\, \mathrm{d}\e\ge\inf_{G\in\mathcal{E}((1-\tilde\gamma(r))\mathrm{mes}_d(\mathcal{G}_r(\by)),\,\by,\,r)}\int_{G}V(\e)\, \mathrm{d}\e\ge\\
&&	\hskip-5mm\frac{\theta-1}{\theta}\delta(r)V^\star(\delta(r),\by,r).
	\end{eqnarray*}
	This estimate, condition \eqref{cndlebesrear} and Theorem \ref{thcondlebesgue} imply the desired claim.
	Theorem  \ref{thcondlebesrear} is proven.
\end{proof}

\subsection{Proof of Theorem \ref{prlogmdensdoubrearr}}

\begin{proof}
Let us take  $r\in(0,1)$ and a cube $Q_{r}(\by)$. It is clear that this cube intersects 
not more than $2^d$ adjacent cubes $Q_1(\vec l_k)\,(\vec l_k\in\Z^d)$  and among them
there is a cube $Q_1(\vec l_k)$ such that for some $\tilde\by\in Q_1(\vec l_k)$: $Q_{\tilde r}(\tilde\by)\subseteq Q_1(\vec l_k)\cap Q_{r}(\by)$
with $\tilde r=r/2$. Thus, for any $\by\in\R^d$ we can choose $\vec l(\by,r)\in Z^d$ such that
\begin{equation}\label{incltointersect}
Q_{\tilde r}(\tilde\by)\subseteq Q_r(\by)\cap Q_1(\vec l(\by,r)).
\end{equation}
Denote
	\begin{equation}\label{dfhatgamrKnew}
	\hat\gamma(\rho,K)=K\gamma(\theta \rho/m^2)\theta^d\quad (K>0).
	\end{equation}
	By Lemma \ref{lminclrear2} and  Lemma \ref{prestbarVstarbelow} with 
	 $W(\e)=V(\e)$, we get that for some  $K>0$ and $j\in\{1,2,\dots,n\}$ with
	$n=\big[\log_m\big(\frac{1}{\theta\tilde r}\big)\big]+2$  
	\begin{eqnarray}\label{ineqforrearV1}
&&V^\star\big(\hat\gamma(\tilde r,K),\,Q_{ r}(\by)
	 \big)\ge\nonumber\\
	 && V^\star\big(\hat\gamma(\tilde r,K),\,Q_{\tilde r}(\tilde\by)
	\big)\ge\min_{\vec\xi:\,Q(\vec\xi,\,n)\subseteq F_j} V^\star\big(\gamma(m^{-n}),\,\vec\xi\,,n),
	\end{eqnarray}
	where $\bar\gamma(r)=\kappa\hat\gamma(r/\sqrt{d},\,K)$ and $F_j$ is a non-empty union of cubes $Q(\vec\xi,\,n)$ with $\vec\xi\in m^{-n}\cdot\Z^d$, such that $F_j\subseteq Q_{\bar r}(\tilde \by)\cap D_j(\vec l(\by,r))$. Notice that since the function $\gamma(r)$ satisfies conditions \eqref{cndtildgam}, the function $\bar\gamma(r)$ satisfies this condition too for some $r_0>0$.
	Then in view of estimate \eqref{ineqforrearV1},  the property \eqref{incltointersect} and condition \eqref{cndlogdensdoubrearrprev}, condition \eqref{cndlebesrear}
	of Theorem  \ref{thcondlebesrear} is satisfied with $\gamma(r)=\bar\gamma(r)$. 
	Then the spectrum of the operator $H=-\Delta+V(\e)\cdot$ is discrete. Theorem \ref{prlogmdensdoubrearr} is proven.
\end{proof}

In the proof of Theorem \ref{prlogmdensdoubrearr} we have used the following claims, proved in \cite{Zel1}:

\begin{lemma}\label{prestbarVstarbelow}[\cite{Zel1}, Lemma 4.3]
	Suppose that in a cube $Q_1(\vec l)\,(\vec l\in\Z^d)$ there is  a sequence of  subsets $\{D_n(\vec l)\}_{n=1}^\infty$ forming in it a $(\log_m,\,\theta)$-dense system.
	Let $\gamma:\,(0,\,r_0)\rightarrow\R$ be a monotone non-increasing function with $r_0=\min\{1,\,1/(m^2\theta)\}$. Then 
	for some  $K>0$ and for
	any cube $Q_r(\by)\subset Q_1(\vec l)$ there are $j\in\{1,2,\dots,n\}$ with
	\begin{equation}\label{dfnlog1}
	n=\big[\log_m\big(\frac{1}{\theta r}\big)\big]+2
	\end{equation}
	and a non-empty set $F_j\subseteq Q_r(\by)\cap D_j(\vec l)$, which is a union of cubes 
	$Q(\vec\xi\,,n)$ with $\vec\xi\in m^{-n}\cdot\Z^d$,
	such that for any nonnegative function $W\in L_1(Q_r(\by))$ the inequality
	\begin{eqnarray}\label{ineqWstaryrxin}
	W^\star(\hat\gamma(r,K),Q_r(\by))\ge
	\min_{\vec\xi:\;Q(\vec\xi,\,n)\subseteq F_j}W^\star\big(\gamma(m^{-n}),\,\vec\xi,\,n)\big) 
	\end{eqnarray}
	is valid, where the function $\hat\gamma(r,K)$ is defined by \eqref{dfhatgamrKnew}.
\end{lemma}

\begin{lemma}\label{lminclrear2}[\cite{Zel1}, Lemma 4.5]
	Let $\Omega_1$ and $\Omega_2$ be measurable subsets of $\R^d$ such that $\Omega_1\subseteq\Omega_2$ and $W(\e)$ be a non-negative measurable function 
	defined on $\Omega_2$. Then for any $t>0$ the inequality
	$W^\star(t,\Omega_1)\le W^\star(t,\Omega_2)$
	is valid.
\end{lemma}

\subsection{Proof of Theorem \ref{thestLagr}}
\begin{proof}
Let us use the reformulation of Problem \ref{extrprobl} given in Remark \ref{rembinprob}.
Denote $u(x)=2\chi_\mathrm{E}(x)-1$. Then it is easy to see that
\begin{equation*}\label{relaxIW}
I_W(t,\,X,\,\mu)\ge(1/2)\inf_{u\in\mathcal{F}(t,X,\,\mu)}\int_X W(x)(1+u(x))\,
\mu(\mathrm{d}x),
\end{equation*}
where 
\begin{eqnarray}\label{dfFtXmu}
&&\mathcal{F}(t,X,\,\mu):=\nonumber\\
&&\{u\in L_2(X,\mu):\;u^2(x)\le1\;\mathrm{a.e.}\;\mathrm{in}\,X,\;\int_X u(x)\,\mu(dx) \ge\sigma(t)\}.
\end{eqnarray}
Consider the Lagrangian:
\begin{eqnarray}\label{dfLagr}
&&L(u,\lambda,\nu)=\int_X W(x)(1+u(x))\,
\mu(\mathrm{d}x)-\int_X \lambda(x)(1-u^2(x))\,\mu(dx)-\nonumber\\
&&\nu\big(\int_X u(x)\,\mu(dx)-\sigma(t)\big),
\end{eqnarray}
where $\nu\ge 0$, $\lambda\in L_2^+(X,\,\mu)=\{\lambda\in L_2(X,\mu):\;\lambda(x)\ge 0\;\mathrm{a.e.}\;\mathrm{in}\,X \}$. It is clear that
\begin{equation*}
I_W(t,\,X,\,\mu)\ge(1/2)\sup_{\lambda\in L_1^+(X,\mu),\,\nu>0}\;\inf_{u\in L_2(X,\mu)}L(u,\lambda,\nu).
\end{equation*}
We shall use an easier version of the Lagrangian relaxation, assuming that $\lambda(x)\equiv\lambda=const>0$ and $\nu>0$. Then
\begin{equation}\label{estIwLagr}
I_W(t,\,X,\,\mu)\ge(1/2)\sup_{\lambda>0,\,\nu>0}m(\lambda,\nu),
\end{equation}
where
\begin{equation}\label{dfmlamnu}
m(\lambda,\nu)=\inf_{u\in L_2(X,\mu)}L(u,\lambda,\nu).
\end{equation}
Let us find $m(\lambda,\nu)$. To this end we shall find critical points of the functional $L(u,\lambda,\nu)$ for fixed $\lambda>0,\,\nu>0$. We see that its Gateaux differential by $u$ is:
\begin{eqnarray*}
&&D_uL(u,\lambda,\nu)(\eta)=\\
&&\int_XW(x)\eta(x)\,\mu(dx)+2\lambda\int_Xu(x)\eta(x)\,\mu(dx)-\nu\int_x\eta(x)\,\mu(dx),
\end{eqnarray*}
where $\eta\in L_2(X,\mu)$. Hence its gradient by $u$ is:
\begin{equation*}
\nabla_uL(u,\lambda,\nu)=W(x)-\nu+2\lambda u(x).
\end{equation*}
Therefore for any fixed $\lambda>0,\,\nu>0$ the functional $L(u,\lambda,\nu)$ has only one critical point in $L_2(X,\mu)$:
\begin{equation*}
u_\star(x,\lambda,\nu)=\frac{\nu-W(x)}{2\lambda}.
\end{equation*}
Consider the second Gateaux differential of $L(u,\lambda,\nu)$ by $u$: 
\begin{equation*}
D_u^2L(u,\lambda,\nu)(\eta,\eta)=2\lambda\int_X(\eta(x))^2\mu(dx)>0,\quad
\mathrm{if}\quad \Vert\eta\Vert_{L_2(X,\mu)}=1.
\end{equation*}
This means that the functional $L(u,\lambda,\nu)$ is strictly convex on $L_2(X,\,\mu)$ for any fixed $\lambda>0,\,\nu>0$, Hence the point $u_\star(x,\lambda,\nu)$ yields the minimum for this  functional.   Therefore, in view of \eqref{dfmlamnu}, we have:
\begin{eqnarray}\label{formformlamnu}
&&\hskip-10mm m(\lambda,\nu)=\int_XW(x)\Big(1+\frac{\nu-W(x)}{2\lambda}\Big)\,\mu(dx)+\nonumber\\
&&\hskip-10mm\lambda\int_X\Big(\frac{(\nu-W(x))^2}{4\lambda^2}-1\Big)\,\mu(dx)-\nu\Big(\int_X\frac{\nu-W(x)}{2\lambda}\,\mu(dx)-\sigma(t)\Big)=\nonumber\\
&&\hskip-10mm \big(1+\frac{\nu}{2\lambda}\big)\mathrm{E}(W)-\frac{1}{2\lambda}\mathrm{E}(W^2)-\lambda+\frac{\nu^2}{4\lambda}-
\frac{\nu}{2\lambda}\mathrm{E}(W)+\frac{1}{4\lambda}\mathrm{E}(W^2)-\nonumber\\
&&\hskip-10mm\frac{\nu^2}{2\lambda}+\frac{\nu}{2\lambda}\mathrm{E}(W)+\nu\sigma(t)=\big(1+\frac{\nu}{2\lambda}\big)\mathrm{E}(W)-\frac{1}{4\lambda}\mathrm{E}(W^2)-\lambda-\frac{\nu^2}{4\lambda}+
\nonumber\\
&&\hskip-10mm\nu\sigma(t)=\mathrm{E}(W)-\frac{1}{4\lambda}\big(\nu-\mathrm{E}(W)\big)^2-\frac{1}{4\lambda}\Big(\mathrm{E}(W^2)-
(\mathrm{E}(W))^2\Big)-\lambda+\nonumber\\
&&\hskip-10mm\nu\sigma(t)
\end{eqnarray}
Now let us find $\sup_{\lambda>0,\,\nu>0}m(\lambda,\nu)$. We see that
\begin{eqnarray}\label{derivmlamnu}
&&\frac{\partial}{\partial\lambda}m(\lambda,\nu)=\frac{1}{4\lambda^2}\Big[\big(\nu-\mathrm{E}(W)\big)^2+
\mathrm{E}(W^2)-(\mathrm{E}(W))^2\Big]-1,\nonumber\\
&&\frac{\partial}{\partial\nu}m(\lambda,\nu)=\frac{1}{2\lambda}\mathrm{E}(W)-\frac{\nu}{2\lambda}+\sigma(t).
\end{eqnarray}
Hence we have the following system of equations for critical points of the function $m(\lambda,\nu)$:
\begin{eqnarray*}
&&\big(\nu-\mathrm{E}(W)\big)^2+
\mathrm{E}(W^2)-(\mathrm{E}(W))^2=4\lambda^2,\nonumber\\
&&\nu-\mathrm{E}(W)=2\lambda\sigma(t).
\end{eqnarray*}
Suppose that $\mathrm{E}(W^2)-(\mathrm{E}(W))^2>0$. We see that in this case the function $m(\lambda,\nu)$ has the unique critical point in the domain $\lambda>0,\,\nu>0$:
\begin{equation*}
\lambda_\star=\frac{\sqrt{\mathrm{E}(W^2)-(\mathrm{E}(W))^2}}{2\sqrt{1-\sigma^2(t)}},\quad \nu_\star=\mathrm{E}(W)+2\lambda_\star\sigma(t),
\end{equation*}
Then in view of \eqref{formformlamnu},
\begin{eqnarray}\label{mlamstnust}
&&\hskip-5mm m(\lambda_\star,\,\nu_\star)=\mathrm{E}(W)-\lambda_\star(1+\sigma^2(t))-\frac{1}{4\lambda_\star}\Big(\mathrm{E}(W^2)-
(\mathrm{E}(W))^2\Big)+\nonumber\\
&&\hskip-5mm\mathrm{E}(W)\sigma(t)+2\lambda_\star\sigma^2(t)=\mathrm{E}(W)(1+\sigma(t))-\lambda_\star(1-\sigma^2(t))-\nonumber\\
&&\hskip-5mm\frac{1}{4\lambda_\star}\Big(\mathrm{E}(W^2)-
(\mathrm{E}(W))^2\Big)=(1+\sigma(t))\mathrm{E}(W)-\\
&&\hskip-5mm\sqrt{1-\sigma^2(t)}\sqrt{\mathrm{E}(W^2)-(\mathrm{E}(W))^2}=(1+\sigma(t))\mathrm{E}(W)-\sqrt{1-\sigma^2(t)}\cdot\mathrm{Dev}(W).\nonumber
\end{eqnarray}
Let us show that $m(\lambda_\star,\,\nu_\star)$ is the maximal value of $m(\lambda,\nu)$ in the domain $\lambda>0,\,\nu>0$. To this end consider the Hessian matrix $\mathrm{Hes}(m)(\lambda,\nu)$ of $m(\lambda,\nu)$ in this domain. Using \eqref{derivmlamnu}, we have:
\begin{eqnarray*}\label{Hessmlannu}
&&-\mathrm{Hes}(m)(\lambda,\nu)=\left(\begin{array}{ll}
-\frac{\partial^2}{\partial\lambda^2}m(\lambda,\nu),&-\frac{\partial^2}{\partial\lambda\partial\nu}m(\lambda,\nu)\\
-\frac{\partial^2}{\partial\lambda\partial\nu}m(\lambda,\nu),&-\frac{\partial^2}{\partial\nu^2}m(\lambda,\nu)
\end{array}\right)=\nonumber\\
&&\left(\begin{array}{ll}
\frac{1}{2\lambda^3}\Big[\big(\nu-\mathrm{E}(W)\big)^2+
\mathrm{E}(W^2)-(\mathrm{E}(W))^2\Big],&\frac{1}{2\lambda^2}(\nu-\mathrm{E}(W))\\
\frac{1}{2\lambda^2}(\mathrm{E}(W)-\nu),&\frac{1}{2\lambda}
\end{array}\right)
\end{eqnarray*}
We see that $-\frac{\partial^2}{\partial\lambda^2}m(\lambda,\nu)> 0$, $-\frac{\partial^2}{\partial\nu^2}m(\lambda,\nu)>0$ and
\begin{equation*}
\det(-\mathrm{Hes}(m)(\lambda,\nu))> 0,
\end{equation*}
i.e., the matrix $-\mathrm{Hes}(m)(\lambda,\nu)$ is positive-definite in any point of the convex domain $\lambda>0$, $\nu>0$. This means that the function $m(\lambda,\nu)$ is strictly concave in this domain. 
 Thus, in the case where $\mathrm{Dev}(W)>0$ the equality holds: $m(\lambda_\star,\nu_\star)=\max_{\lambda>0,\nu>0}m(\lambda,\nu)$. Hence in view of 
\eqref{estIwLagr} and \eqref{mlamstnust} the inequality \eqref{estLagr} is valid. It is clear that if
$\mathrm{Dev}(W)=0$, then $W(x)=\mathrm{E}(W)$ for $\mu$-almost all $x\in X$. In view of definition \eqref{dfIVOm}, in this case $I_W(t,\,X,\,\mu)\ge t\cdot \mathrm{E}(W)$. Since $\sigma(t)=2t-1$, the inequality \eqref{estLagr} is valid also in this case. Theorem \ref{thestLagr} is proven.
\end{proof}

\subsection{Proof of Theorem \ref{thdiscLagr}}
\begin{proof}
	In view of estimate \eqref{estLagr} (Theorem \ref{thestLagr}) with $t=1-\gamma(r)/2$ and $\sigma(t)=2t-1=1-\gamma(r)$,
\begin{eqnarray*}
&&\inf_{F\in\mathcal{M}_{\gamma(r)/2}(\by,r)}\int_{\mathcal{G}_r(\by)\setminus
	F}V(\e)\, \mathrm{d}\e=\mathrm{mes}_d\big(\mathcal{G}_r(\by)\big)\times\\
&&\inf_{F\in\mathcal{M}_{\gamma(r)/2}(\by,r)}\int_{\mathcal{G}_r(\by)\setminus
	F}V(\e)\,m_{\by,r}(\mathrm{d}\e)\ge\\
&&\mathrm{mes}_d\big(\mathcal{G}_r(\by)\big)\frac{1+\sigma(t)}{2}\Big(\mathrm{E}_{\by,r}(V)-\sqrt{\frac{1-\sigma(t)}{1+\sigma(t)}}\cdot\mathrm{Dev}_{\by,r}(V)\Big)\ge\\
&&\mathrm{mes}_d\big(\mathcal{G}_r(\by)\big)(1-\gamma(r)/2)\Big(\mathrm{E}_{\by,r}(V)-\sqrt{\gamma(r)}\cdot\mathrm{Dev}_{\by,r}(V)\Big).
\end{eqnarray*}
This estimate, condition \eqref{cnddisclarg} and Theorem \ref{thcondlebesgue} (with $\gamma(r)/2$ instead of $\gamma(r)$) imply the desired claim. Theorem \ref{thdiscLagr} is proven. 
\end{proof}

\subsection{Proof of Theorem \ref{thperturb}}

\begin{proof}
In view of definition \eqref{defDsW} and
	condition \eqref{sigmain} there exists $R>0$ such that for any
	$\by\in\R^d\setminus B_R(0)$ there is a vector field
	$\Gamma_{\by,r}\in {\mathcal D}_{\by,\,r}(W)$ such that
	$\Vert \Gamma_{\by,r}\Vert_d\le\sigma/(2C(d))$. Taking into
	account that
	$W(\e)=\mathrm{div}\vec\Gamma_{\by,r}(\e)\;(\e\in Q_r(\by))$
	and using claim (ii) of Proposition \ref{lmest}, we obtain:
	\begin{eqnarray*}
		\int_{Q_r(\by)}W(\e)|u(\e)|^2\, \mathrm{d}\e\le
		\sigma\int_{Q_r(\by)}|\nabla u(\e)|^2\,\mathrm{d}\e,
	\end{eqnarray*}
	hence
	\begin{eqnarray*}
		&&\hskip-5mm(H u,u)_{Q_r(\by)}=\int_{Q_r(\by)}\Big(|\nabla
		u(\e)|^2+\big(V_0(\e)+W(\e)\big)|u(\e)|^2\Big)|u(\e)|^2\,\mathrm{d}\e\ge\nonumber\\
		&&\hskip-5mm\big(1-\sigma\big)\int_{Q_r(\by)}|\nabla u(\e)|^2\,
		\mathrm{d}\e+\int_{Q_r(\by)}V_0(\e)|u(\e)|^2\,\mathrm{d}\e=(H_\sigma
		u,u)_{Q_r(\by)}.
		\end{eqnarray*}
	This inequality and Proposition \ref{prloc} imply the desired
	claims.	Theorem \ref{thperturb} is proven.
\end{proof}

\subsection{Proof of Theorem \ref{thFour}}
\begin{proof}
	Consider the step function
\begin{equation}\label{dfnonpert}
V_0(\e)=\int_{Q_1(\vec l)}V(\bs)\,\mathrm{d}\bs\quad\mathrm{for}\quad\e\in Q_1(\vec l).
\end{equation}
In view of condition \eqref{intQitendinfty},	$\lim_{|\e|\rightarrow\infty}V_0(\e)=\infty$. Hence the spectrum of operator $H_\theta=-\theta\Delta+V_0(\e)\cdot\;(\theta>0)$ is discrete \cite{Gl}. Denote $W(\e)=V(\e)-V_0(\e)$. We see from \eqref{dfnonpert} that for any $\vec l\in\Z^d$ $\int_{Q_1(\vec l)}W(\e)\,\mathrm{d}\e=0$. Furthermore, since $v\in L_{2,\,loc}(\R^d)$, then $W\in L_{2,\,loc}(\R^d)$.
Hence in view of condition \eqref{cndtransfour}
and Proposition \ref{prdiveq}, for any $\vec l\in\Z^d$ divergence equation $\mathrm{div}\vec\Gamma=W$
has in the cube $Q_1(\vec l)$ a potential solution $\Gamma_{\vec l}(\e)$ belonging to the space $L_d(Q_1(\vec l))$ and 
$\limsup_{|\vec l|\rightarrow\infty}\Vert\vec \Gamma_{\vec l}\Vert_d<\frac{1}{2^{d+1}C(d)}$. 
Notice that if $r_0\in(0,\,1)$, any cube $Q_{r_0}(\by)$ intersects at most $2^d$ cubes $Q_1(\vec l)\,(\vec l\in\Z^d)$. Hence perturbation $W(\e)$ of the potential $V_0(\e)$
satisfies condition \eqref{cndpert} of Theorem \ref{thperturb}, therefore the spectrum of H is discrete and bounded below. Theorem \ref{thFour} is proven.
\end{proof}

\subsection{Proof of Theorem \ref{thlattsupp}}
\begin{proof}
(i) Consider the step functiom $V_0(\e)$, defined by \eqref{dfnonpert} and denote $W(\e)=V(\e)-V_0(\e)$.	In view of condition \eqref{cndNtimesrd} and Remark \ref{remonint}, condition \eqref{intQitendinfty} of Theorem \ref{thFour} is fulfilled. Furthermore, conditions \eqref{SinL2}, \eqref{intSposit}, definitions \eqref{persumS}, \eqref{dfVrbysumm} imply that $W\in L_{2,\,loc}(\R^d)$.
 Using dilation rule for the Fourier transform $\mathcal{F}$ on $\T^d$ \cite{Kam}, we have that for the function $V_r(\e)$, defined by \eqref{persumS}, and a natural $m$ the equality
\begin{eqnarray*}
\mathcal{F}\big(V_r(m\cdot\e)\big)(\vec k)=\left\{\begin{array}{ll}
\mathcal{F}\big(V_r\big)(\vec k/m)& \mathrm{for}\quad\vec k\in m\cdot\Z^d,\\
0& \mathrm{otherwise}
\end{array}
\right.
\end{eqnarray*}
is valid. Furthermore, applying Poisson formula for Fourier transform on $\T^d$ of the periodic summation and using dilation and translation rules for Fourier transform on $\R^d$ \cite{Kam}, we have:
\begin{equation*}
\mathcal{F}\big(V_r\big)(\vec k)=\hat S\big((\cdot-\vec c)/r\big)(\vec k)=r^d\exp\big(-i2\pi\langle\vec c,\,\vec k\rangle\big)\hat S(r\vec k).
\end{equation*}
Denote by $W_{\vec l}(\e)$ and $V_{\vec l}(\e)$ the $1$-periodic continuations (by all the variables) of of the functions $W\vert_{Q_1(\vec l)}$ and $V\vert_{Q_1(\vec l)}$ from the cube $Q_1(\vec l)$ to the whole $\R^d$.
Then in view of condition \eqref{inclHr}-\eqref{dfHrveck}, 
\begin{eqnarray*}
&&\Big\Vert \frac{\mathcal{F}\big(W_{\vec l}\big)(\vec k)}{2\pi i|\vec k|^2}\vec k\Big\Vert_{l_{d^\prime}}=\Big\Vert\frac{\mathcal{F}\big(V_{\vec l}\big)(\vec k)}{2\pi i|\vec k|^2}\vec k\Big\Vert_{l_{d^\prime}(\Z^d\setminus\{\vec 0\})}=\nonumber\\
&&\mathcal{N}(\vec l)(r(\vec l))^d\Big(\sum_{\vec k\in m(\vec l)\cdot\Z^d\setminus\{\vec 0\} }\Big(\frac{|\hat S\big(r(\vec l)\vec k/m(\vec l)\big)|}{2\pi|\vec k|}\Big)^{d^\prime}\Big)^{1/d^\prime}=\nonumber\\
&&\frac{\mathcal{N}(\vec l)(r(\vec l))^d}{m(\vec l)}\Big(\sum_{\vec k\in\Z^d\setminus\{\vec 0\} }\Big(\frac{|\hat S\big(r(\vec l)\vec k\big)|}{2\pi|\vec k|}\Big)^{d^\prime}\Big)^{1/d^\prime}\le\frac{\mathcal{N}(\vec l)}{m(\vec l)}\bar H<\infty.
\end{eqnarray*}
Then condition \eqref{cndfracNlml} implies that the functions  $W_{\vec l}(\e)\,(\vec l\in\Z^d)$ 
satisfy condition \eqref{cndtransfour} of Theorem \ref{thFour}. Hence the spectrum of $H$ is discrete. Claim (i) is proven.

(ii) Let us take $S(\e)=\chi_{1/2}(\e)$. It is known that Fourier transform of the characteristic function of a ball is expressed via Bessel function (\cite{Gr}, p. 578). In particular, 
\begin{equation*}
\hat\chi_{1/2}(\vec\omega)=A(d)\frac{1}{(\pi|\vec\omega|)^{d/2}}J_{d/2}(\pi|\vec\omega|).,
\end{equation*}
where $A(d)=\frac{2\sqrt{\pi}v_{d-1}\Gamma\big((d-1)/2\big)}{2^{d/2}}$ and $v_{d-1}$ is the volume of $d-1$-dimensional unit ball. It is known  that for $\nu>0$ Bessel function $J_\nu(x)$ is continuous in $[]0,\,+\infty)$ 
and the following asymptotic formula is valid for $x\rightarrow+\infty$:
\begin{equation*}
J_\nu(x)=\sqrt{\frac{2}{\pi x}}\cos\big(x-\pi\nu/2-\pi/4\big)+O\big(\frac{1}{x^{3/2}}\big)
\end{equation*}
(\cite{Gr}. pp.168, 580). Hence there is a constant $\bar J(\nu)>0$  such that $|J_\nu(x)|\le\frac{\hat J(\nu)}{\sqrt{x}}$ for any $x>0$. 
Therefore for some constant $B(d)>0$ and any $\vec k\in\Z^d\setminus\{\vec 0\}$, $r\in(0,\,1)$
\begin{equation*}
r^d\frac{\hat S(r\vec k)}{2\pi|\vec k|}\le B(d)\frac{r^{(d-1)/2}}{|\vec k|^{(d+3)/2}}\le\frac{B(d)}{|\vec k|^{(d+3)/2}}.
\end{equation*}
Since the series $\sum_{\vec k\in\Z^d\setminus\{\vec 0\}}\Big(\frac{1}{|\vec k|^{(d+3)/2}}\Big)^{d^\prime} $ converges for $d<5$, condition \eqref{inclHr}
 is satisfied for $d\in\{3,\,4\}$. Claim (ii) is proven. Hence Theorem \ref{thlattsupp} is proven too.
 \end{proof} 

\subsection{Proof of Theorem \ref{thmeasinstcapRVOm}}

\begin{proof}
	Assume that $F\in\mathcal{N}_{\tilde\gamma(r)}(\by,r,\Omega)$, where $\tilde\gamma(r)=\big(\gamma(r)/G\big)^{(d-2)/d}$ and $G:=\Big(\frac{\bar r}{r_m}\Big)^d$, $\bar r:=\max_{\omega\in S^{d-1}}r(\omega)$, $r_m:=\min_{\omega\in S^{d-1}}r(\omega)$.  Notice that inclusion \eqref{inclGminOm},
	is equivalent to $E\subseteq\bar{\mathcal{G}}_r(\by)\cap\Omega$, where $E=\bar{\mathcal{G}}_r(\by)\setminus F$. Furthermore, in  the proof of Theorem \ref{thcondlebesgue}  we have shown using the isocapacity inequality that  condition \eqref{defNcap}	implies $\mathrm{mes}_d(F)\le\gamma(r)\mathrm{mes}_d(\mathcal{G}_r(\by))$, i.e., $\mathrm{mes}_d(E)\ge(1-\gamma(r))\mathrm{mes}_d(\mathcal{G}_r(\by))$. In particular, these arguments imply that if $\mathrm{mes}_d(\mathcal{G}_r(\by)\cap\Omega)<(1-\gamma(r))\,\mathrm{mes}_d(\mathcal{G}_r(\by))$, then 
	$\mathcal{N}_{\tilde\gamma(r)}(\by,r,\Omega)=\emptyset$, i.e.,
	\begin{equation*}
	\inf_{F\in\mathcal{N}_{\tilde\gamma(r)}(\by,r,\Omega)}\int_{{\mathcal G}_r(\by)\setminus
		F}V(\e)\, \mathrm{d}\e=+\infty.
	\end{equation*}
	Then in view of definition \eqref{dfRWyrOm}, these circumstances mean that
	\begin{equation*}
	\inf_{F\in\mathcal{N}_{\tilde\gamma(r)}(\by,r,\Omega)}\int_{{\mathcal G}_r(\by)\setminus
		F}V(\e)\, \mathrm{d}\e\ge R_V(\by,r,\Omega,\gamma(r)).
	\end{equation*}
	This estimate, condition \eqref{cndmeasinstcapcalGOm} and Theorem \ref{thMazShOm} imply the desired claim. Theorem \ref{thmeasinstcapRVOm} is proven
\end{proof}

\subsection{Proof of Corollary \ref{cordiscspecLaplace}}

\begin{proof}
	It is clear that the function $\tilde\gamma(r)=\gamma(r)/2$ satisfies condition \eqref{cndtildgam}. Furthermore, we have from \eqref{cnddiscspecLap}:
	\begin{equation*}
	\limsup_{|\by|\rightarrow\infty}\frac{\mathrm{mes}_d(\mathcal{G}_r(\by)\cap\Omega)}{\mathrm{mes}_d(\mathcal{G}_r(\by))}< 1-\tilde\gamma(r). 
	\end{equation*}
	This circumstance, definition \eqref{dfRWyrOm} and Theorem \ref{thmeasinstcapRVOm} imply the desired claim. Corollary \ref{cordiscspecLaplace} is proven.
\end{proof}

\subsection{Proof of Theorem \ref{thdiscLagrOm}}

\begin{proof}
Assume that $\mathrm{mes}_d(\mathcal{G}_r(\by)\cap\Omega)>(1-\gamma(r)/2)\,\mathrm{mes}_d(\mathcal{G}_r(\by))$.
In view of estimate \eqref{estLagr} (Theorem \ref{thestLagr}) with $t=1-\bar\gamma(\gamma(r),\Omega)$ and $\sigma(t)=2t-1=1-2\bar\gamma(\gamma(r),\Omega)$.
\begin{eqnarray*}
	&&\hskip-8mm\inf_{E\in\mathcal{P}_{\gamma(r)/2}(\by,r,\Omega)}\int_E V(\e)\,\mathrm{d}\e=\mathrm{mes}_d\big(\mathcal{G}_r(\by)\cap\Omega\big)\times\\
	&&\hskip-8mm\inf_{E\in\mathcal{P}_{\gamma(r)/2}(\by,r,\Omega)}\int_E V(\e)\,m_{d,r,\Omega}(\mathrm{d}\e)\ge\\
	&&\hskip-8mm\mathrm{mes}_d\big(\mathcal{G}_r(\by)\cap\Omega\big)\frac{1+\sigma(t)}{2}\Big(\mathrm{E}_{\by,r,\Omega}(V)-\sqrt{\frac{1-\sigma(t)}{1+\sigma(t)}}\cdot\mathrm{Dev}_{\by,r,\Omega}(V)\Big)\ge\\
	&&\hskip-8mm\mathrm{mes}_d\big(\mathcal{G}_r(\by)\big)(1-\gamma(r)/2)\Big(\mathrm{E}_{\by,r,\Omega}(V)-\sqrt{2\bar\gamma(\gamma(r),\Omega)}\cdot\mathrm{Dev}_{\by,r,\Omega}(V)\Big).
\end{eqnarray*}
Notice that if $\mathrm{mes}_d(\mathcal{G}_r(\by)\cap\Omega)=(1-\gamma(r)/2)\,\mathrm{mes}_d(\mathcal{G}_r(\by))$ (i.e., 
$\bar\gamma(\gamma(r),\Omega)=0$), then
\begin{equation*}
\mathrm{E}_{\by,r,\Omega}(V)=\frac{\int_{\mathcal{G}_r(\by)\cap\Omega}V(\e)\,\mathrm{d}\e}{(1-\gamma(r)/2)\,\mathrm{mes}_d(\mathcal{G}_r(\by))}
\end{equation*}
These circumstances, condition \eqref{cnddisclargOm} and Theorem \ref{thmeasinstcapRVOm} (with $\gamma(r)/2$) instead of $\gamma(r)$) imply the desired claim. Theorem \ref{thdiscLagrOm} is proven.
\end{proof}

\section{Some examples} \label{sec:examples}
\setcounter{equation}{0}
First of all, let us recall some examples of the $(\log_m,\,\theta)$-dense system (Definition \ref{dfdenslogmtetpart1}), constructed in \cite{Zel1}.

\begin{example}\label{ex1}
	Consider the classical middle third Cantor set $\mathcal{C}\subset[0,1]$, Let $I_{n,k}\;(n=1,2,\dots),\,k=1,2,\dots, 2^{n-1}$ be the closures 
	of intervals adjacent to $\mathcal{C}$. It is known that they are pairwise disjoint and for any fixed $n$ and each $k\in\{1,2,\dots, 2^{n-1}\}$ 
	$\mathrm{mes}_1(I_{n,k})=3^{-n}$. For fixed $n$ we shall number the intervals $I_{n,k}$ from the left to the right. Denote $D_n=\bigcup_{k=1}^{2^{n-1}}I_{n,k}$. 
	As we have shown in \cite{Zel1} (Example 5.1),  the sequence $\{D_n\}_{n=1}^\infty$ forms 
	in $[0,1]$ a $(\log_3,\,1/9)$-dense system. 
\end{example}

\begin{example}\label{ex2}
	Consider a cube $Q_1(\by)\subset\R^d$, represented in the form $Q_1(\by)=Q_1(\by_1)\times Q_1(\by_2)$, where 
	$Q_1(\by_1)\subset\R^{d_1}$ and $Q_1(\by_2)\subset\R^{d_2}$. Let $\{D_n\}_{n=1}^\infty$ be a sequence of subsets of 
	the cube $Q_1(\by_1)$, forming in it a $(\log_m,\theta)$-dense system. It is easy to see that the sequence
	$\{D_n\times Q_1(\by_2)\}_{n=1}^\infty$ forms in $Q_1(\by)$ a $(\log_m,\theta)$-dense system too.	
\end{example}

Now we shall consider counterexamples connected with conditions of discreteness of the spectrum
of the operator $H=-\Delta+V(\e)\cdot$, obtained above.

\begin{example}\label{ex4}
	In \cite{Zel1} we have constructed an example of the potential $V(\e)\ge 0$ which satisfies conditions 
	of Theorem \ref{prlogmdensdoubrearr} (hence the spectrum of the operator $H=-\Delta+V(\e)\cdot$ is discrete) , but the condition \eqref{cndGMD} of the criterion from \cite{GMD},
	formulated in Remark \ref{remGMD}, is not satisfied for it. In view of the technical mistake, admitted in  \cite{Zel1} and pointed in Remark \ref{remmistake}, we shall recall and correct this construction.
	Let us return to the sequence $\{D_n\}_{n=1}^\infty$ of subsets of the interval $[0,\,1]$ and considered in Example \ref{ex1}, and the following 
	sequence of subsets of the cube $Q_1(0)$:
	\begin{equation}\label{dfcalD}
	\mathcal{D}_n=D_n\times[0,\,1]^{d-1}.
	\end{equation} 
	Consider also the translations of the cube $Q_1(0)$ and the sets
	$\mathcal{D}_n$ by the vectors $\vec l=(l_1,l_2,\dots,l_d)\in\Z^d$: $Q_1(\vec l)=Q_1(0)+\{\vec l\}$, 
	\begin{equation}\label{dfcalDvecl}
	\mathcal{D}_n(\vec l)=\mathcal{D}_n+\{\vec l\}.
	\end{equation}
	The arguments of Examples \ref{ex1} and \ref{ex2} imply that for any fixed $\vec l\in\Z^d$  the sequence $\{\mathcal{D}_n(\vec l)\}_{n=1}^\infty$ 
	forms in $Q_1(\vec l)$ a $(\log_3,\,1/9)$-dense system.  For $\beta\in(0,1)$ consider on $\R$ the
	$1$-periodic function $\theta_\beta(x)$, defined on the interval
	$(0,1]$ in the following manner:
	\begin{equation}\label{dfthatbetnew}
	\theta_\beta(x)=\left\{\begin{array}{ll}
	1&\quad\mathrm{for}\quad x\in(0,\beta],\\
	0&\quad\mathrm{for}\quad x\in(\beta,1].
	\end{array}\right.
	\end{equation}
	Let us take
	\begin{equation}\label{alphain}
	\alpha\in\Big(0,\,2d/(d-2)\Big)
	\end{equation} 
	(in  \cite{Zel1} we have imposed the more restrictive condition $\alpha\in(0,\,2)$). Consider the following function, defined on $(0,1]$:
	\begin{eqnarray}\label{dfSigmaaNalphnew}
	\Sigma_{N,\,p,\alpha}(x):=\left\{\begin{array}{ll}
	0&\quad\mathrm{for}\quad x\in(0,\,1]\setminus\bigcup_{n=1}^\infty D_n,\\
	N\theta_\beta(3^px)\vert_{\beta=3^{-\alpha n}}&\quad\mathrm{for}\quad x\in
	D_n\;(n=1,2,\dots)
	\end{array}\right.
	\end{eqnarray}
	$(N>0,\,p\in\N)$. 
	Consider a function
	$\mathcal{N}:\Z^d\rightarrow\R_+$, satisfying the condition
	\begin{equation}\label{cndNl1new}
	\mathcal{N}(\vec l)\ge 1,\quad\mathcal{N}(\vec
	l)\rightarrow\infty\quad\mathrm{for}\quad |\vec
	l|_\infty\rightarrow\infty.
	\end{equation}
Let us construct the
	desired potential in the following manner:
	\begin{eqnarray}\label{dfpotentValnew}
	&&V_\alpha(\e):=\Sigma_{N,\,p,\,\alpha}(P_1(\e-\vec
	l))\vert_{N=\mathcal{N}(\vec l),\,p=|\vec l|_\infty+1}\quad
	\mathrm{for}\quad\vec l\in \Z^d\\
	&&\mathrm{and}\quad \e\in
	Q_1(\vec l),\nonumber
	\end{eqnarray}
	where the operator $P_1$ is defined in the following manner:
	\begin{equation}\label{dfP1}
	\mathrm{for}\quad
	\e=(x_1,x_2,\dots,x_d)\quad P_1\e:=x_1.
	\end{equation}
	It is clear that $V_\alpha\in L_{\infty,\,loc}(\R^d)$.  
	Let us prove that the potential $V_\alpha(\e)$ satisfies all the conditions of Theorem \ref{prlogmdensdoubrearr}. Let us take 
	a natural $n$, 
	\begin{equation}\label{jin1n}
	j\in\{1,2,\dots,n\}
	\end{equation}
	and a cube  
	\begin{equation}\label{inclcube}
	Q(\vec\xi,\,n)\subseteq\mathcal{D}_j(\vec l)
	\end{equation} 
	of the $3$-adic partition of $Q_1(\vec l)$. Let us notice that 
	\begin{equation*}
	P_1\big(Q(\vec\xi,\,n)\big)=[k\,3^{-n},\,(k+1)\,3^{-n}]
	\end{equation*}
	for some $k\in\Z$.
	Then taking into account definitions \eqref{dfcalD}-\eqref{dfSigmaaNalphnew}, \eqref{dfpotentValnew} and the $1$-periodicity of $\theta_\beta(t)$, we get for $|\vec l|_\infty>n$:
	\begin{eqnarray}\label{bigest}
	&&\mathrm{mes}_d\Big(\big\{\e\in Q(\vec\xi,n):\,V_\alpha(\e)>0\big\}\Big)=\\
	&&3^{-(d-1)n}\mathrm{mes}_1\Big(\big\{x\in[k\,3^{-n},\,(k+1)\,3^{-n}]:\,\nonumber\\
	&&\theta_\beta(3^p\,x)\vert_{\beta=3^{-\alpha j},\,p=|\vec l|_\infty+1}>0 \big\}\Big)=\nonumber\\
	&&3^{-(d-1)n}\mathrm{mes}_1\Big(\big\{x\in[0,\,3^{-n}]:\,\theta_\beta(3^p\,x)\vert_{\beta=3^{-\alpha j},\,p=|\vec l|_\infty+1}>0 \big\}\Big)=\nonumber\\
	&&\frac{3^{-(d-1)n}}{3^{|\vec l|_\infty+1}}\mathrm{mes}_1\Big(\big\{t\in[0,\,3^{|\vec l|_\infty+1-n}]:\,\theta_\beta(t)\vert_{\beta=3^{-\alpha j}}>0\big\}\Big)=\nonumber\\
	&& 3^{-dn}\mathrm{mes}_1\Big(\big\{t\in[0,\,1]:\,\theta_\beta(t)\vert_{\beta=3^{-\alpha j}}>0\big\}\Big)=\nonumber\\
	&&3^{-\alpha j}\mathrm{mes}_d\big(Q(\vec\xi,\,n)\big)\ge 3^{-\alpha n}\mathrm{mes}_d\big(Q(\vec\xi,\,n)\big).\nonumber
	\end{eqnarray}
	Therefore in view of definitions \eqref{dfSVyrdel}-\eqref{dfLVsry}, \eqref{dfSigmaaNalphnew} and \eqref{dfpotentValnew},
	\begin{equation*}
	V_\alpha^\star\big(3^{-\alpha n},\,\vec\xi,\,n\big)\ge\mathcal{N}(\vec l),
	\end{equation*}
	if conditions \eqref{jin1n} and \eqref{inclcube} are satisfied.  This estimate and conditions \eqref{alphain}, \eqref{cndNl1new}-b imply that condition 
	\eqref{cndlogdensdoubrearrprev} of Theorem \ref{prlogmdensdoubrearr} is satisfied for the potential $V_\alpha(\e)$ with $\gamma(r)=r^\alpha$ satisfying condition \eqref{cndtildgam}. Hence the spectrum of the operator 
	$H=-\Delta+V_\alpha(\e)\cdot$ is discrete.
\end{example}

\begin{example}\label{exLagr}
We shall construct an example of the potential $V(\e)\ge 0$ which satisfies conditions 
of Theorem \ref{thdiscLagr} (hence the spectrum of the operator $H=-\Delta+V(\e)\cdot$ is discrete) , but the condition \eqref{cndGMD} of the criterion from \cite{GMD},
formulated in Remark \ref{remGMD}, is not satisfied for it.	Consider the following partition of the interval $[0,1]$:  $[0,1]=\bigcup_{n=1}^\infty I_n$, where 
\begin{equation}\label{dfIn}
I_n=(2^{-n},\,2^{-n+1}]\quad(n=1,2,3,\dots)..
\end{equation}
  Let us take 
\begin{equation}\label{cndforalpha}
\alpha\in(0,\;2d/(d-2))
\end{equation}
and	consider the following function, defined on $(0,1]$:
\begin{eqnarray}\label{dfPsialNLagr}
\hskip10mm\Psi_{N,\,p,\alpha}(x):= N\theta_\beta(px)\vert_{\beta=2^{-\alpha n}}\quad\mathrm{for}\quad x\in I_n\;(n=1,2,\dots), 
\end{eqnarray}
where $\theta_\beta(x)$ is the $1$-periodic function, defined on $(0,1])$ by \eqref{dfthatbetnew}.
Let us construct the
desired potential in the following manner:
\begin{eqnarray}\label{dfpotentValnew1}
&&V_\alpha(\e):=\Psi_{N,\,p,\,\alpha}(P_1(\e-\vec
l))\vert_{N=\mathcal{N}(\vec l),\,p=|\vec l|_\infty+1}\quad
\mathrm{for}\quad\vec l\in \Z^d\\
&&\mathrm{and}\quad \e\in
Q_1(\vec l),\nonumber
\end{eqnarray}
where the operator $P_1$ is defined by \eqref{dfP1} and the function $\mathcal{N}(\vec l)$ has the properties \eqref{cndNl1new}. It is clear that $V_\alpha\in L_{\infty,\,loc}(\R^d)$. Let us show that the potential $V_\alpha(\e)$ satisfies condition  \eqref{cnddisclarg} of Theorem \ref{thdiscLagr}. Consider a cube $Q_r(\by)\;(r\in(0,1])$. 
Assume initially that 
\begin{equation*}
Q_r(\by)\subseteq Q_1(\vec l)\quad (\by=(y_1,y_2,\dots,y_d),\;\vec l=(l_1,l_2,\dots,l_d)). 
\end{equation*}
Denote  
\begin{equation}\label{dfFgamma}
F_\delta(x)=x-\sqrt{\delta}\cdot\sqrt{x-x^2}\quad(x\in(0,1],\;\delta>0)
\end{equation}
and 
\begin{equation}\label{dfKralpha}
\gamma(r)=K r^\alpha,
\end{equation}
 where $K>0$ will be chosen in the sequel. Consider the probability measure on the cube $Q_r(\by)$:  
\begin{equation*}
m_{d,r}(A):=\frac{\mathrm{mes}_d(A)}{\mathrm{mes}_d(Q_r(\by))}\quad(A\in\Sigma_L(Q_r(\by))).
\end{equation*} 
Recall that we denote by $\mathrm{E}_{\by,r}(W)$ and $\mathrm{Dev}_{\by,r}(W)$ the expectation and deviation of a real-valued  function $W\in L_2(Q_r(\by),\, m_{d,r})$ (definitions \eqref{dfedxpec}, \eqref{dfVar}).
In view of definitions  \eqref{dfPsialNLagr}, \eqref{dfpotentValnew1}, \eqref{dfFgamma} and Lemma \ref{lmpropsigmaNbetnew},  we have:
\begin{eqnarray}\label{expresFgam}
&&\hskip-15mm\mathrm{E}_{\by,r}(V_\alpha)-\sqrt{\gamma(r)}\cdot\mathrm{Dev}_{\by,r}(V_\alpha)=\nonumber\\
&&\hskip-15mm\mathcal{N}(\vec l)\cdot F_{\gamma(r)}\Big(\sum_{n=1}^\infty\mu_{d,r}\big(\{\e\in Q_r(\by)\cap D_n(\vec l):\;V_{\alpha}(\e)>0\}\big)\Big)=\nonumber\\
&&\hskip-15mm\mathcal{N}(\vec l)\cdot F_{\gamma(r)}\Big(\sum_{n=1}^\infty2^{-\alpha n}\frac{\mathrm{mes}_1\big([y_1-l_1,\,y_1-l_1+r]\cap I_n\big)+\theta(n,\vec l)}{r}\Big), 
\end{eqnarray}
where $[y_1,\,y_1+r]\subseteq [l_1,\,l_1+1]$ and the quantity $\theta(n,\vec l)$ satisfies the estimate:
\begin{equation*}\label{estthetnvecl}
|\theta(n,\vec l)|\le\frac{2\cdot 2^{-\alpha n}}{|\vec l|_\infty+1}.
\end{equation*}
Let $m_g$ be the measure on $[0,\,1]$, whose density $g(x)$ with respect to $\mathrm{mes}_1$ is
\begin{equation}\label{dfgx}
g(x)=2^{-\alpha n}\quad\mathrm{for}\quad x\in I_n.
\end{equation}
Then we have:
\begin{equation}\label{exprsumbymeas}
\sum_{n=1}^\infty 2^{-\alpha n}\mathrm{mes}_1\big([y_1-l_1,\,y_1-l_1+r]\cap I_n\big)=m_g\big([y_1-l_1,\,y_1-l_1+r]\big).
\end{equation}
On the other hand, by claim (i) of Lemma \ref{propofFdelta},
\begin{eqnarray}\label{estFgambelow}
&&\hskip-6mmF_{\gamma(r)}\Big(\frac{1}{r}m_g\big([y_1-l_1,\,y_1-l_1+r]\big)+
2^{-\alpha n}\frac{\theta(n,\vec l)}{r}\Big)\ge\\
&&\hskip-6mm F_{\gamma(r)}\Big(\frac{1}{r}m_g\big([y_1-l_1,\,y_1-l_1+r]\big)\Big)-
\big(1+\sqrt{2\gamma(r)}\big)\sqrt{\frac{2}{r(1-2^{-\alpha})(|\vec l|_\infty+1)}}\nonumber.
\end{eqnarray}
Notice that in view of \eqref{dfgx}, \eqref{dfIn}, the function $g(x)$ is monotone nondecreasing. Hence
\begin{eqnarray*}
	&&m_g\big([y_1-l_1,\,y_1-l_1+r]\big)=\int_{y_1-l_1}^{y_1-l_1+r}g(x)\,\mathrm{d}x\ge\nonumber\\
	&&\int_{0}^{r}g(x)\,\mathrm{d}x=m_g\big([0,r]\big)=
\sum_{n=1}^\infty 2^{-\alpha n}\mathrm{mes}_1\big([0,r]\cap I_n\big).
\end{eqnarray*}
Let $n_0$ be the natural number such that that $2^{-n_0}\le r<2\cdot 2^{-n_0}$.
Then
\begin{eqnarray*}
&&\sum_{n=1}^\infty
2^{-\alpha n}\frac{\mathrm{mes}_1\big([0,r]\cap I_n\big)}{r}\ge
\sum_{n=n_0}^\infty2^{-\alpha n}\frac{2^{-n+1}}{r}=\frac{2}{r}\cdot\frac{2^{-(\alpha+1)n_0}}{1-2^{-\alpha-1}}\ge\nonumber\\
&&\frac{2}{r}\cdot\frac{(r/2)^{-(\alpha+1)}}{1-2^{-\alpha-1}}=\frac{2}{2^{\alpha+1}-1}r^\alpha.
\end{eqnarray*}
The last estimates, definition \eqref{dfKralpha} and claims (ii), (iii) of Lemma \ref{propofFdelta} imply that if 
$0<K<\frac{2}{2^{\alpha+1}-1}$, then 
\begin{equation*}
F_{\gamma(r)}\Big(\frac{2}{2^{\alpha+1}-1}r^\alpha\Big)>0
\end{equation*}
and
\begin{eqnarray*}
F_{\gamma(r)}\Big(\frac{1}{r}m_g\big([y_1-l_1,\,y_1-l_1+r]\big)\Big)\ge F_{\gamma(r)}\Big(\frac{2}{2^{\alpha+1}-1}r^\alpha\Big). 
\end{eqnarray*}
Hence in view of \eqref{expresFgam}, \eqref{exprsumbymeas} and \eqref{estFgambelow},
\begin{eqnarray*}
&&\mathrm{E}_{\by,r}(V_\alpha)-\sqrt{\gamma(r)}\cdot\mathrm{Dev}_{\by,r}(V_\alpha)\ge\\
&&\mathcal{N}(\vec l)\cdot\Big(F_{\gamma(r)}\Big(\frac{2}{2^{\alpha+1}-1}r^\alpha\Big)-\big(1+\sqrt{2\gamma(r)}\big)\sqrt{\frac{2}{r(1-2^{-\alpha})(|\vec l|_\infty+1)}}\Big).\nonumber
\end{eqnarray*}
Now consider the case where the inclusion $Q_r(\by)\subseteq Q_1(\vec l)$ is not valid. It is clear that  the cube $Q_r(\by)$ intersects 
not more than $2^d$ adjacent cubes $Q_1(\vec l_k)\,(\vec l_k\in\Z^d,\,k=1,2,\dots,\nu)$  and among them
there is a cube $Q_1(\vec l_k)$ such that for some $\tilde\by\in Q_1(\vec l_k)$: $Q_{ r/2}(\tilde\by)\subseteq Q_1(\vec l_k)\cap Q_{r}(\by)$. Thus, for any $\by\in\R^d$ we can choose $\tilde{\vec l}\in Z^d$ such that
$Q_{r/2}(\tilde\by)\subseteq Q_r(\by)\cap Q_1(\tilde{\vec l})$.
Using again  definitions  \eqref{dfPsialNLagr}, \eqref{dfpotentValnew1}, \eqref{dfFgamma} and Lemma \ref{lmpropsigmaNbetnew},  we get:
\begin{eqnarray*}
&&\mathrm{E}_{\by,r}(V_\alpha)-\sqrt{\gamma(r)}\cdot\mathrm{Dev}_{\by,r}(V_\alpha)=\\
&&F_{\gamma(r)}\Big(\sum_{k=1}^\nu\mathcal{N}(\vec l_k)\sum_{n=1}^\infty\mu_{d,r}\big(\{\e\in Q_r(\by)\cap D_n(\vec l_k):\;V_{\alpha}(\e)>0\}\big)\Big)
\end{eqnarray*}
and
\begin{eqnarray*}
&&\sum_{k=1}^\nu\mathcal{N}(\vec l_k)\sum_{n=1}^\infty\mu_{d,r}\big(\{\e\in Q_r(\by)\cap D_n(\vec l_k):\;V_{\alpha}(\e)>0\}\big)\ge\\
&&\mathcal{N}(\tilde{\vec l})\cdot\frac{1}{r}m_g\big([\tilde y_1-\tilde l_1,\,\tilde y_1-\tilde l_1+r/2]\big)-
\max_{k\in\{1,2,\dots,\nu\}}\frac{2}{r(1-2^{-\alpha})(|\vec l_k|_\infty+1)},
\end{eqnarray*}
where $\tilde y_1=P_1(\tilde\by)$ and $\tilde l_1=P_1(\tilde{\vec l})$. Using the previous arguments, we can choose $K>0$ in \eqref{dfKralpha} small enough, such that 
\begin{equation*}
\mathrm{E}_{\by,r}(V_\alpha)-\sqrt{\gamma(r)}\cdot\mathrm{Dev}_{\by,r}(V_\alpha)\ge \mathcal{N}(\tilde{\vec l})\cdot\Big(\Phi(r)-\max_{k\in\{1,2,\dots,\nu\}}\frac{\phi(r)}{\sqrt{|\vec l_k|_\infty+1}}\Big),
\end{equation*}
where the quantities $\Phi(r)$ and $\phi(r)$ depend only on $r$ and they are positive. This estimate and condition \eqref{cndNl1new} imply that the potential $V_\alpha(\e)$ satisfies condition \eqref{cnddisclarg}. Furthermore, in view of \eqref{dfKralpha} and \eqref{alphain}, the function $\gamma(r)$ satisfies conditions \eqref{cndtildgam}. Thus, all the conditions of Theorem \ref{thdiscLagr}
are satisfied for the potential $V_\alpha(\e)$, hence the spectrum of operator $H=-\Delta+V_\alpha(\e)\cdot$ is discrete.

Let us show that the condition \eqref{cndGMD} of the criterion from \cite{GMD} is not satisfied for the potential $V_\alpha(\e)$. To this end 
for any natural $n$ consider $\vec l_n\in\Z^d$ such that $|\vec l_n|_\infty=n$ and take a cube
\begin{equation*}
 Q(\vec\xi_n,\,n+1)=Q_{2^{-n-1}}(\vec\xi_n)\subset I_n\times[0,1]^{d-1}+\{\vec l_n\}.
\end{equation*}
 In view 
of \eqref{dfthatbetnew}, \eqref{dfSigmaaNalphnew}, \eqref{dfpotentValnew} and \eqref{cndNl1new}-a, $\int_{Q(\vec\xi_n,\,n)}V_\alpha(\e)\,\mathrm{d}\e>0$.
Then taking $\delta>0$, we have:
\begin{eqnarray*}
	&&\hskip-5mm\big\{\e\in Q(\vec\xi_n,\,n+1):\,V_\alpha(\e)\ge\frac{\delta}{\mathrm{mes}_d(Q(\vec\xi_n,\,n+1))}\int_{Q(\vec\xi_n,\,n+1)}V_\alpha(\e)\,\mathrm{d}\e\big\}\subseteq\\
	&&\hskip-5mm\big\{\e\in Q(\vec\xi_n,\,n+1):\,V_\alpha(\e)>0\big\}.
\end{eqnarray*}
On the other hand, in view of definitions  \eqref{dfPsialNLagr}, \eqref{dfpotentValnew1} and Lemma \ref{lmpropsigmaNbetnew},  we get:
\begin{eqnarray*}
&&\mathrm{mes}_d\Big(\big\{\e\in Q(\vec\xi_n,\,n+1):\,V_\alpha(\e)>0\big\}\Big)\le\\ 
&&2^{-\alpha n}\Big(1+\frac{2}{n+1}\Big)\mathrm{mes}_d\big(Q(\vec\xi_n,\,n+1)\big).
\end{eqnarray*}
These circumstances mean that the sequence 
\begin{equation*}
\frac{\mathrm{mes}_d\Big(\big\{\e\in Q(\vec\xi_n,\,n+1):\,V_\alpha(\e)\ge\frac{\delta}{\mathrm{mes}_d(Q(\vec\xi_n,\,n+1))}\int_{Q(\vec\xi_n,\,n+1)}V_\alpha(\e)\,\mathrm{d}\e\big\}\Big)}
{\mathrm{mes}_d\big(Q(\vec\xi_n,\,n+1)\big)}
\end{equation*}
tends to zero as $n\rightarrow\infty$. This means that condition \eqref{cndGMD} is not satisfied for the potential $V_\alpha(\e)$.

\end{example}

In the construction of Example \ref{exLagr} we have used the following claims:

\begin{lemma}\label{lmpropsigmaNbetnew}
	For an interval $[a,b]$ and $S>0$ consider the quantity
	\begin{equation*}
	M(S,\beta,a,b)=\mathrm{mes}_1\{x\in[a,b]:\,\theta_\beta(Sx)>0\},
	\end{equation*}
	where $\theta_\beta(x)$ is defined by
	\eqref{dfthatbetnew}. It satisfies the inequalities
	\begin{equation}\label{estaboveMNnew}
	M(S,\beta,a,b)\le \beta(b-a)+2\beta/S,
	\end{equation}
	\begin{equation}\label{estbelowMNnew}
	M(S,\beta,a,b)\ge \beta(b-a)-2\beta/S;
	\end{equation}
\end{lemma}
\begin{proof}
	Taking into account the $1$-periodicity of $\theta_\beta(x)$, we obtain:
	\begin{eqnarray*}
		&&M(S,\beta,a,b)\le M(S,\beta,[S a]/S,([S b]+1)/S)=\\
		&&S^{-1}M(1,\beta,[S a],[S b]+1)=\frac{\beta}{S}\big([S]+1-[S a]\big)\le\\
		&&\beta(b-a)+2\beta/S,
	\end{eqnarray*}
	\begin{eqnarray*}
		&&M(S,\beta,a,b)\ge M(S,\beta,([S a]+1)/S,[S b]/S)=\\
		&&S^{-1}M(1,\beta,[S a]+1,[S b])
		=\frac{\beta}{S}\big([S b]-1-[S a]\big)\ge\\
		&&\beta(b-a)-2\beta/S.
	\end{eqnarray*}
\end{proof}

\begin{lemma}\label{propofFdelta}
	The function $F_\delta(x)$, defined by \eqref{dfFgamma} has the properties:
	
\noindent (i) for any $h\in(0,\,1)$ and $x\in(0, 1-h)$
	\begin{equation*}
	|F_\delta(x+h)-F_\delta(x)|\le \big(1+\sqrt{2\delta}\big)\sqrt{h};
	\end{equation*}
	\vskip2mm
\noindent (ii) it is positive in  $(\frac{\delta}{1+\delta},\,1]$;
	\vskip2mm
\noindent (iii) it is increasing in  $(\frac{\delta}{2(1+\delta+\sqrt{1+\delta})},\,1]$.
	\begin{proof}
		(i) We have:
	\begin{equation}\label{derivFdelta}
	F_\delta^\prime(x)=1-\sqrt{\delta}\frac{1-2x}{2\sqrt{x-x^2}}.
	\end{equation}
	Then
\begin{equation}\label{estdiffer}
|F_\delta(x+h)-F_\delta(x)|\le\int_x^{x+h}|F_\delta^\prime(s)|\,\mathrm{d}s\le h+\frac{\sqrt{\delta}}{2}\int_x^{x+h}\frac{\mathrm{d}s}{\sqrt{s}\cdot\sqrt{1-s}}.
\end{equation}
Consider the case where $x\in[0,\,1/2]$. Then 
\begin{eqnarray*}\label{estint1}
&&\int_x^{x+h}\frac{\mathrm{d}s}{\sqrt{s}\cdot\sqrt{1-s}}\le\sqrt{2}\int_x^{x+h}\frac{\mathrm{d}s}{\sqrt{s}}=2\sqrt{2}\big(\sqrt{x+h}-\sqrt{x}\big)=\nonumber\\
&&2\sqrt{2}\frac{h}{\sqrt{x+h}+\sqrt{x}}\le 2\sqrt{2h}.
\end{eqnarray*}
Now consider the case where $x\in(1/2,\,1-h]$. We have:
\begin{eqnarray*}\label{estint2}
&&\hskip-5mm\int_x^{x+h}\frac{\mathrm{d}s}{\sqrt{s}\cdot\sqrt{1-s}}\le\sqrt{2}\int_x^{x+h}\frac{\mathrm{d}s}{\sqrt{1-s}}=2\sqrt{2}\big(\sqrt{1-x}-\sqrt{1-h-x}\big)=\nonumber\\
&&\hskip-5mm2\sqrt{2}\frac{h}{\sqrt{1-x}+\sqrt{1-h-x}}\le 2\sqrt{2h}.
\end{eqnarray*}
These estimates imply claim (i).

(ii) The representation $F_\delta(x)=\frac{x^2-\delta(x-x^2)}{x+\sqrt{\delta}\cdot\sqrt{x-x^2}}$
implies that $F_\delta(x)>0$ if and only if $x\in(\frac{\delta}{1+\delta},\,1]$. Claim (ii) is proven.

(iii) We see from \eqref{derivFdelta} that $F_\delta^\prime(x)>0$, if $x\in[1/2,\,1]$. Consider the case where $x\in[0,\,1/2)$. The representation 
\begin{equation*}
F_\delta^\prime(x)=\frac{2\sqrt{x-x^2}-\sqrt{\delta}(1-2x)}{2\sqrt{x-x^2}}=\frac{4(x-x^2)-\delta(1-2x)^2}{2\sqrt{x-x^2}\big(2\sqrt{x-x^2}+\sqrt{\delta}(1-2x)\big)}
\end{equation*}
implies that $F_\delta^\prime(x)>0$, if and only if $x^2-x+\frac{\delta}{4(1+\delta)}<0$, i.e.,
\begin{equation*}
x\in\Big(\frac{\delta}{2(1+\delta+\sqrt{1+\delta})},\,1/2\Big).
\end{equation*}
 Claim (iii) is proven.
\end{proof}
\end{lemma}

\begin{example}\label{exlagreasier}
	The following example shows that in some cases condition \eqref{cnddisclarg} of Theorem \ref{thdiscLagr} is easier verifiable than condition \eqref{cndlebesrear} of Theorem \ref{thcondlebesrear}. Consider the potential
\begin{equation*}\label{dfpoteasver}
V(\e)=\mathcal{N}(\vec l)W_{\vec l}(\e)\quad\mathrm{for}\quad \e=(x_1,x_2,\dots.x_d)\in Q_1(\vec l)\quad(\vec l\in\Z^d),
\end{equation*}
where 
\begin{equation*}\label{dfWvecl}
W_{\vec l}(\e)=\prod_{j=1}^d\big(1-\cos(m(\vec l)x_j)\big),
\end{equation*}
$\mathcal{N}(\vec l)>0$, $m(\vec l)> 0$  and 
\begin{equation}\label{cndforNm}
\mathcal{N}(\vec l)\rightarrow\infty,\quad m(\vec l)\rightarrow\infty\quad\mathrm{for}\quad |\vec l|\rightarrow\infty.
\end{equation}
Consider the case where $Q_r(\by)\subseteq Q_1(\vec l)\,(
r\in(0,1/2),\by=(y_1,y_2,\dots,y_d),\,\vec l\in\Z^d)$. We have:
\begin{equation}\label{transNl}
\mathrm{E}_{\by,r}(V)-\sqrt{\gamma}\cdot\mathrm{Dev}_{\by,r}(V)=\mathcal{N}(\vec l)\big(\mathrm{E}_{\by,r}(W_{\vec l})-\sqrt{\gamma}\cdot\mathrm{Dev}_{\by,r}(W_{\vec l})\big)
\end{equation}
$(\gamma>0)$,
\begin{eqnarray*}
&&\mathrm{E}_{\by,r}(W_{\vec l})=\frac{1}{r^d}\int_{Q_r(\by)}W_{\vec l}(\e)\,\mathrm{d}\e=\prod_{j=1}^d\frac{1}{r}\int_{y_j}^{y_j+r}\big(1-\cos(m(\vec l)x_j)\big)\,\mathrm{d}x_j=\nonumber\\
&&\prod_{j=1}^d\Big(1+O\big(\frac{1}{m(\vec l)}\big)\Big)
\end{eqnarray*}
and 
\begin{eqnarray*}
&&\hskip-5mm\mathrm{E}_{\by,r}(W_{\vec l}^2)=\frac{1}{r^d}\int_{Q_r(\by)}W_{\vec l}^2(\e)\,\mathrm{d}\e=\prod_{j=1}^d\frac{1}{r}\int_{y_j}^{y_j+r}\big(1-\cos(m(\vec l)x_j)\big)^2\,\mathrm{d}x_j=\nonumber\\
&&\hskip-5mm\prod_{j=1}^d\Big(3/2+O\big(\frac{1}{m(\vec l)}\big)\Big)
\end{eqnarray*} 
for $|\vec l|\rightarrow\infty$. Hence 
\begin{equation*}
\big(\mathrm{Dev}_{\by,r}(W_{\vec l})\big)^2=\mathrm{E}_{\by,r}(W_{\vec l}^2)-\big(\mathrm{E}_{\by,r}(W_{\vec l})\big)^2=\big(3/2\big)^d-1+O\big(\frac{1}{m(\vec l)}\big).
\end{equation*}
and 
\begin{eqnarray}\label{finest}
\mathrm{E}_{\by,r}(W_{\vec l})-\sqrt{\gamma}\cdot\mathrm{Dev}_{\by,r}(W_{\vec l})=1-\sqrt{\gamma}\cdot\sqrt{\big(3/2\big)^d-1}+O\big(\frac{1}{m(\vec l)}\big).
\end{eqnarray}
In the case, where the inclusion $Q_r(\by)\subseteq Q_1(\vec l)$ is not valid for any $\vec l\in\Z^d$, the cube $Q_r(\by)$ intersects not more than $2^d$ adjacent cubes $Q_1(\vec l_k)\,(\vec l_k\in\Z^d)$. Making analogous asymptotic estimates for integrals $\int_{Q_r(\by\cap Q_1(\vec l_k)}W_{\vec l_k}(\e)\,\mathrm{d}\e$ and $\int_{Q_r(\by\cap Q_1(\vec l_k)}W_{\vec l_k}^2(\e)\,\mathrm{d}\e$ and summing each of them by $k$, we obtain the asymptotic estimate with the same leading term as in \eqref{finest}.
From \eqref{cndforNm}-\eqref{finest} we obtain that if $\gamma\in\Big(0,\,\frac{1}{\big(3/2\big)^d-1}\Big)$, the potential $V(\e)$ satisfies condition \eqref{cnddisclarg} of Theorem \ref{thdiscLagr} with $Q_r(\by)$ instead of $\mathcal{G}_r(\by)$ and $\gamma(r)\equiv\gamma$. Hence the spectrum of operator $H=-\Delta+V(\e)\cdot$ is discrete. Notice that in view of Remark \ref{remLagrimplrear}, conditions of Theorem \ref{thdiscLagr} imply conditions of Theorem \ref{thcondlebesrear}. But the immediate application of Theorem \ref{thcondlebesrear} to this example needs technically complicated geometrical considerations for calculation of the non-increasing rearrangement of the potential $V(\e)$. 
\end{example}

\begin{example}\label{exrearbutnotLagr}
The potential $V(\e)$, considered below,  satisfies the conditions of Theorem \ref{thcondlebesrear},  but condition \eqref{cnddisclarg} of Theorem \ref{thdiscLagr} is not satisfied for it. We put 
\begin{equation}\label{Vexpxquad}
V(\e)= e^{|\e|^2}.
\end{equation}
Since $\lim_{|\e|\rightarrow\infty}V(\e)=+\infty$, it is clear that $V(\e)$ satisfies the conditions of Theorem \ref{thcondlebesrear}. On the other hand, the representation 
\begin{eqnarray*}
&&\mathrm{E}_{\by,r}(V)-\sqrt{\gamma(r)}\cdot\mathrm{Dev}_{\by,r}(V)=\\
&&-\sqrt{\mathrm{E}_{\by,r}(V^2)}\Big(\sqrt{\gamma(r)}\sqrt{1-\frac{(\mathrm{E}_{\by,r}(V))^2}{\mathrm{E}_{\by,r}(V^2)}}-\frac{\mathrm{E}_{\by,r}(V)}{\sqrt{\mathrm{E}_{\by,r}(V^2)}}\Big)
\end{eqnarray*}
implies that condition \eqref{cnddisclarg} is not satisfied, if 
\begin{equation}\label{limfraczero}
\lim_{|\by|\rightarrow\infty}\frac{(\mathrm{E}_{\by,r}(V))^2}{\mathrm{E}_{\by,r}(V^2)}=0.
\end{equation}
Let us take $\mathcal{G}_r(\by)=Q_{r_1}(\by-r\vec a)$ 
($r_1=2r/\sqrt{d}$, $\vec a=(d^{-1/2}, d^{-1/2},\dots,d^{-1/2})$). Using twice L'Hospital's rule, we have:
\begin{eqnarray*}
&&\lim_{|y|\rightarrow\infty}\frac{\big(\int_y^{y+r_1}e^{x^2}\,\mathrm{d}x\big)^2}{\int_y^{y+r_1}e^{2x^2}\,\mathrm{d}x}=2\lim_{|y|\rightarrow\infty}\frac{\big(\int_y^{y+r_1}e^{x^2}\,\mathrm{d}x\big)\big(e^{(y+r_1)^2}-e^{y^2}\big)}{e^{2(y+r_1)^2}-e^{2y^2}}=\\
&&\lim_{|y|\rightarrow\infty}\frac{e^{(y+r_1)^2}-e^{y^2}}{(y+r_1)e^{(y+r_1)^2}+ye^{y^2}}=0.
\end{eqnarray*}
This result and definition \eqref{Vexpxquad} imply equality \eqref{limfraczero}. Thus, the considered potential does not satisfy condition \eqref{cnddisclarg}. 
\end{example}

\begin{example}\label{ex5}
	We shall construct an example of the potential $V(\e)$, for which
	the conditions of claim (ii) of Theorem \ref{thperturb} are satisfied (hence the
	spectrum of the operator $H=-\Delta+V(\e)\cdot$ is discrete), but
	condition \eqref{cndlebesrear} of Theorem \ref{thcondlebesrear} is not satisfied for it.
	Suppose that the $(\log_3,\,1/9)$-dense systems $\mathcal{D}_n(\vec l)$ in $Q_1(\vec l)$ $(\vec l\in\Z^d)$ and the potential $V_\alpha(\e)$ 
	are the same as in Example \ref{ex4}. Assume that
	\begin{equation}\label{addcondforalph}
	\log_3 2<\alpha<2d/(d-2).
	\end{equation}
	Consider a monotone non-decreasing function $\psi:\,(0,\,1)\rightarrow(0,\,1)$ satisfying the condition
	\begin{equation}\label{cndforpsi}
	\lim_{r\downarrow 0}r^{-2d/(d-2)}\psi(r)<\infty
	\end{equation}
	and the function
	\begin{eqnarray}\label{dfSigmaaNalphnewest}
	&&\hskip-5mm\tilde\Sigma_{N,\,p,\,\psi}(x):=\nonumber\\
	&&\hskip-5mm\left\{\begin{array}{ll}
	0&\quad\mathrm{for}\quad x\in[0,\,1]\setminus\bigcup_{n=1}^\infty D_n,\\
	N\frac{3^{-\alpha n}}{\beta}\theta_\beta(3^px)\vert_{\beta=\psi(3^{-(n+1)})}&\quad\mathrm{for}\quad x\in
	D_n\;(n=1,2,\dots)
	\end{array}\right.
	\end{eqnarray}
	$(N>0,\,p\in\N)$, where $\theta_\beta(x)$ is defined by \eqref{dfthatbetnew}.
Consider the potential 
	\begin{eqnarray}\label{dfpotentValnewest}
	&&V_\psi(\e):=\tilde\Sigma_{N,\,p,\,\psi}(P_1(\e-\vec
	l))\vert_{N=\mathcal{N}(\vec l),\,p=|\vec l|_\infty+1}\quad
	\mathrm{for}\quad\vec l\in \Z^d\nonumber\\
	&&\mathrm{and}\quad \e\in
	Q_1(\vec l),
	\end{eqnarray}
	where the function $\mathcal{N}(\vec l)$ satisfies conditions \eqref{cndNl1new}. 
	Recall that the operator
	$P_1$  is defined by \eqref{dfP1}.
	Let $I_{n,k}\,(1\le k\le 2^{n-1})$ be the intervals forming the set $D_n$, defined in Example \ref{ex4}. 
	Taking into account the left inequality of \eqref{addcondforalph} and using Lemma \ref{lmpropsigmaNbetnew}, we have:
	\begin{eqnarray*}
		&&\hskip-5mm\int_{Q_1(\vec l)}V_\psi(\e)\,d\e=
		\mathcal{N}(\vec l)\sum_{n=1}^\infty\sum_{k=1}^{2^{n-1}}\int_{I_{n,k}}\frac{3^{-\alpha n}}{\beta}\theta_\beta(3^{|\vec l|_\infty+1}x)\vert_{\beta=\psi(3^{-(n+1)})}\,dx\le\\
		&&\hskip-5mm\mathcal{N}(\vec l)\sum_{n=1}^\infty 2^{n-1}\Big(3^{-(\alpha+1)n}+\frac{2}{3^{|\vec l|_\infty+1}}3^{-\alpha n}\Big)<\infty.
	\end{eqnarray*}
	Hence $V_\psi\in L_{1,\,loc}(\R^d)$. Let us prove that the potential $V_\psi(\e)$ satisfies the conditions  of claim (ii) of Theorem \ref{thperturb}. 
	Denote $W_{\psi,\alpha}(\e)=V_\psi(\e)-V_\alpha(\e)$. Then, in view of \eqref{dfpotentValnew},
	\begin{eqnarray}\label{dfWpsialph}
	&&W_{\psi,\alpha}(\e):=\hat\Sigma_{N,\,p,\,\psi,\,\alpha}(P_1(\e-\vec
	l))\vert_{N=\mathcal{N}(\vec l),\,p=|\vec l|_\infty+1}\quad
	\mathrm{for}\quad\vec l\in \Z^d\nonumber\\
	&&\mathrm{and}\quad \e\in
	Q_1(\vec l),
	\end{eqnarray}
	where
	\begin{eqnarray}\label{dfhatSigmaapsialph}
	&&\hat\Sigma_{N,\,p,\,\psi,\,\alpha}(x):=\nonumber\\
	&&\left\{\begin{array}{ll}
	0&\quad\mathrm{for}\quad x\in[0,\,1]\setminus\bigcup_{n=1}^\infty D_n,\\
	N\tilde\theta_{\psi,\alpha,n}(3^px)&\quad\mathrm{for}\quad x\in
	D_n\;(n=1,2,\dots),
	\end{array}\right.
	\end{eqnarray}
	and
	\begin{equation}\label{dftildthetpsialphn}
	\tilde\theta_{\psi,\alpha,n}(t)=\frac{3^{-\alpha n}}{\beta}\theta_\beta(t)\vert_{\beta=\psi(3^{-(n+1)})}-\theta_\delta(t)\vert_{\delta=3^{-\alpha n}}.
	\end{equation}
	Then, in view of \eqref{dfthatbetnew},
	\begin{equation}\label{proptildthat}
	\int_0^1\tilde\theta_{\psi,\alpha,n}(t)\,dt=0.
	\end{equation}
	Consider the following vector field in the cube $Q_1(\by)\,(\by\in\R^d)$:
	\begin{equation}\label{dfvecGammanewest}
	\vec\Gamma_{\by}(\e)=\Big(\int_{P_1\by}^{P_1\e}W_{\psi,\alpha}(\xi,\,(I-P_1)\e)\,d\xi,\,0,\dots,0\Big).
	\end{equation}
	By Proposition \ref{prdiv}, it is a generalized solution of the equation $\mathrm{div}(\vec\Gamma_{\by}(\e))=W_{\psi,\alpha}(\e)$ in the cube $Q_1(\by)$.
	Consider the cubes $\{Q_1(\vec l_i)\}_{i=1}^K\;(\vec l_i\in\Z^d)$ intersecting the cube $Q_1(\by)$. It is clear that 
	$K\le 2^d$.   
	Denote $J(\by,i)=[P_1\by,\,P_1\by+1]\cap [P_1\vec l_i,\,P_1\vec l_i+1]$, $J(\by,i,n,k)=J(\by,i)\cap (I_{n,k}+P_1\vec l_i)$. Then, in view of 
	definitions \eqref{dfWpsialph}, \eqref{dfhatSigmaapsialph} and \eqref{dftildthetpsialphn}, property \eqref{proptildthat} and Lemma \ref{lmintoftildthet},
	we have  for $\e\in Q_1(\by)\cap Q_1(\vec l_i)$
	\begin{eqnarray}\label{estmodvecGamma}
	&&|\vec\Gamma_{\by}(\e)|=\Big|\int_{J(\by,i)}W_{\psi,\alpha}(\xi,\,(I-P_1)\e)\,d\xi\Big|\le\nonumber\\
	&&\sum_{n=1}^\infty\sum_{k=1}^{2^{n-1}}\Big|\int_{J(\by,i.n,k)}W_{\psi,\alpha}(\xi,\,(I-P_1)\e)\,d\xi\Big|=\nonumber\\
	&&\mathcal{N}(\vec l_i)\sum_{n=1}^\infty\sum_{k=1}^{2^{n-1}}\Big|\int_{J(\by,i.n,k)}\tilde\theta_{\psi,\alpha,n}(3^p\xi)\vert_{p=|\vec l_i|_\infty+1}\,d\xi\Big|\le\nonumber\\
	&&\frac{2\mathcal{N}(\vec l_i)}{3^{|\vec l_i|_\infty+1}}\sum_{n=1}^\infty\sum_{k=1}^{2^{n-1}}\int_0^1|\tilde\theta_{\psi,\alpha,n}(t)|\,dt\le
	\frac{2\mathcal{N}(\vec l_i)}{3^{|\vec l_i|_\infty+1}}\sum_{n=1}^\infty\Big(\frac{2}{3^\alpha}\Big)^n.
	\end{eqnarray}
	The left inequality of \eqref{addcondforalph} implies that the last series in \eqref{estmodvecGamma} converges. 
	Assume that along with conditions \eqref{cndNl1new} the function $\mathcal{N}(\vec l)$ satisfies the condition
	\begin{equation}\label{addcndfotNl}
	\frac{\mathcal{N}(\vec l)}{3^{|\vec l|_\infty}}\rightarrow 0\quad\mathrm{as}\quad |\vec l|_\infty\rightarrow\infty, 
	\end{equation}
	Then estimate \eqref{estmodvecGamma} implies that
	\begin{equation*}
\Vert\vec \Gamma\Vert_{L_d(Q_r(\by))}\le
\sup_{\e\in Q_1(\by)}|\vec\Gamma_{\by}(\e)|\rightarrow 0\quad\mathrm{as}\quad|\by|\rightarrow\infty,
	\end{equation*}
	hence condition \eqref{cndpert}  is satisfied with $W=W_{\psi,\alpha}$.
	On the other hand, we have shown in Example \ref{ex4} that the right inequality of \eqref{addcondforalph} and condition \eqref{cndNl1new} imply 
	discreteness of the 
	spectrum of  operator $-\Delta+V_\alpha(\e)\cdot$. These circumstances imply that the conditions  
	of  claim (ii) of Theorem \ref{thperturb} are satisfied for the potential $V_\psi(\e)$. Thus, under conditions \eqref{cndNl1new}, \eqref{addcondforalph} and \eqref{addcndfotNl} 
	the spectrum of operator $-\Delta+V_\psi(\e)\cdot$ is discrete. 
	
	Let us show that the potential $V_\psi(\e)$ does not satisfy condition \eqref{cndlebesrear} of Theorem \ref{thcondlebesrear}.
	Let us take a function $\hat\gamma(r)$ satisfying condition \eqref{cndtildgam} for some $r_0>0$. Then by \eqref{cndforpsi},
$\limsup_{r\downarrow 0}\frac{\hat\gamma(r)}{\psi(r)}=\infty$.
Hence for some positive decreasing sequence $\{r_j\}_{j=1}^\infty$ tending to zero
	\begin{equation}\label{decrseqrj}
	\lim_{j\rightarrow \infty}\frac{\hat\gamma(r_j)}{\psi(r_j)}=\infty.
	\end{equation}
	Let us choose an increasing sequence of natural numbers $\{n_j\}_{j=1}^\infty$ such that 
	\begin{equation}\label{ineqforrj}
	3^{-(n_j+1)}\le r_j<3^{-n_j}.
	\end{equation}
	Consider the vectors $\vec l\in\Z^d$ of the form $\vec l=(l,0,\dots,0)$, where $l\in\Z$ will be chosen below. Consider the intervals $I_{n_j,1}=[a_{n_j,1},\,b_{n_j,1}]\subset D_{n_j}$ and
	the cubes $Q_{j,l}=Q_{3^{-n_j}}(\tilde\by_{j,l})$, where $\tilde\by_{j,l}=\big(a_{n_j,1}+l,\,0,\dots,0\big)$. Then $P_1(Q_{j,l})=I_{n_j,1}+l$. In view of the right inequality
	of \eqref{ineqforrj}, $Q_{r_j}(\tilde\by_{j,l})\subset Q_{j,l}$. Then using definitions \eqref{dftildthetpsialphn} and \eqref{dfhatSigmaapsialph}, Lemma \ref{lmpropsigmaNbetnew}, the left inequality of \eqref{ineqforrj} and taking $l=n_j$, we have:
	\begin{eqnarray}\label{estmespositVpsi}
	&&\mathrm{mes}_d\big(\{\e\in Q_{r_j}(\tilde\by_{j,l}):\,V_\psi(\e)>0\}\big)\le
	\mathrm{mes}_d\big(\{\e\in Q_{j,l}:\,V_\psi(\e)>0\}\big)=\nonumber\\
	&&3^{-n_j(d-1)}\mathrm{mes}_1\big(\{x\in I_{n_j,1}:\,\theta_\beta(3^{l+1}x)\vert_{\beta=\psi(3^{-(n_j+1)})}\}\big)\le\nonumber\\
	&&\psi(3^{-(n_j+1)})(3^{-n_jd}+3\cdot 3^{-n_j(d-1)}3^{-l-1})=\nonumber\\
	&&2\psi(3^{-(n_j+1)})3^{-(n_j+1)d}3^d\le 2\cdot 3^d\psi(r_j)\mathrm{mes}_d(Q_{r_j}(\tilde\by_{j,l})).
	\end{eqnarray} 
	On the other hand, in view of \eqref{decrseqrj}, 
	\begin{equation*}
	\exists\, J>0\quad\forall\, j\ge J:\quad 2\cdot 3^d\psi(r_j)<\hat\gamma(r_j).
	\end{equation*}
	This circumstance, estimate \eqref{estmespositVpsi} and definition 
	\eqref{dfSVyrdel}-\eqref{dfLVsry} imply that 
	\begin{equation*}
	\forall\,j\ge J:\quad \bar V_\psi^\star\big(\hat\gamma(r_j)\mathrm{mes}_d(Q_{r_j}(\by_j)),\,Q_{r_j}(\by_j)\big)=0, 
	\end{equation*}
	where $\by_j=\tilde\by_{j,n_j}$. This means that the potential $V_\psi(\e)$ does not satisfy condition \eqref{cndlebesrear}.
	\end{example}

In the construction of Example \ref{ex5} we have used the following claim:

\begin{lemma}\label{lmintoftildthet}
	Let $\tilde\theta:\,\R\rightarrow\R$ be $1$-periodic locally integrable function such that $\int_0^1\tilde\theta(t)\,dt=0$. Then for any 
	interval $[a,b]$ and $S>0$ the inequality
	\begin{equation*}
|\int_a^b\tilde\theta(Sx)\,dx|\le\frac{2}{S}\int_0^1|\tilde\theta(t)|\,dt
	\end{equation*}
	is valid.
\end{lemma}
\begin{proof}
	The representation
	\begin{eqnarray*}
		&&\int_a^b\tilde\theta(Sx)\,dx=\frac{1}{S}\int_{Sa}^{Sb}\tilde\theta(t)\,dt=\frac{1}{S}\Big(\int_{[Sa]}^{[Sb]}\tilde\theta(t)\,dt-
		\int_{[Sa]}^{Sa}\tilde\theta(t)\,dt+\\
		&&\int_{[Sb]}^{Sb}\tilde\theta(t)\,dt\Big)=
		\frac{1}{S}\Big(-
		\int_{[Sa]}^{Sa}\tilde\theta(t)\,dt+\int_{[Sb]}^{Sb}\tilde\theta(t)\,dt\Big)
	\end{eqnarray*}
	implies the desired estimate.
\end{proof}

\appendix

\section{Some classes  of solutions of the divergence equation}
\label{sec:A}

In this section we consider a class of generalized
solutions of the divergence equation
\begin{equation}\label{diveq}
\mathrm{div}\vec\Gamma=W
\end{equation}
in the cube $Q_r(\by)$ Here $W$ is a distribution on $Q_r(\by)$. Recall
that a vector distribution $\vec\Gamma$ is called a
{\it generalized solution} of the above equation, if it satisfies \eqref{diveq} in the distributional sense. In particular, this means, that if $W\in L_1(Q_r(\by))$ and $\vec\Gamma\in L_1(Q_r(\by)$, then
\begin{eqnarray}\label{defgensol}
&&\forall\,\phi\in
C_0^\infty(Q_r(\by)):\nonumber\\
&&\int_{Q_r(\by)}W(\e)\phi(\e)\,
\mathrm{d}\e=-\int_{Q_r(\by)}\langle\vec\Gamma(\e),\,\nabla\phi(\e)\rangle\,
\mathrm{d}\e,
\end{eqnarray}
(\cite{Bour-Br}, \cite{Pau-Pf}). 

\subsection{Reduction to ODE}

Suppose that $W\in L_1(Q_r(\by))$. We shall look for a solution of equation
(\ref{diveq}) of the form:
\begin{equation}\label{frmsol}
\vec\Gamma(\e)=u(\e)\,\vec g(\e)\quad (\e=(x_1,\,x_2,\,\dots\,x_d)),
\end{equation}
where $\vec g(\e)$ is a given vector field on $Q_r(\by)\;(\by=(y_1,\,y_2,\,\dots\,y_d))$  and $u(\e)$ is a unknown
scalar function. Let us take the constant vector field 
\begin{equation}\label{constvecfield}
\vec g(\e)=(1,\,0,\dots,\,0).
\end{equation}
Then (\ref{diveq}) is reduced to
$\frac{\partial u}{\partial x_1}=W(\e)$. 
Consider the solution of last equation in $Q_r(\by)$, satisfying the initial condition 
$u(\e)\vert_{x_1=y_1}=0$.
It has the form
\begin{equation}\label{solscaleq}
u(\e)=\int_{y_1}^{x_1}W(s,\e^\prime)\,\mathrm{d}s\quad(\e^\prime=(x_2,\,x_3,\,\dots,\,x_d)).
\end{equation}
The following claim is valid:

\begin{proposition}\label{prdiv} 
	Let $\vec\Gamma(\e)$ be the vector field, defined by \eqref{frmsol}, \eqref{constvecfield} and \eqref{solscaleq}.
	 \vskip2mm
\noindent (i) If $W\in C^1(Q_r(\by))$, then $\vec\Gamma(\e)$ is a classical solution of divergence equation \eqref{diveq} in $Q_r(\by)$;
	\vskip2mm
	
\noindent (ii) The linear operator $U(W)(\e):=\int_{y_1}^{x_1}W(s,\e^\prime)\,\mathrm{d}s$, taking part in the right hand side of \eqref{solscaleq} acts in the space $L_p(Q_r(\by))\,(p\ge 1)$ and it is bounded;
	\vskip2mm
	
\noindent (iii) If $W\in L_p(Q_r(\by))$, then $\vec\Gamma(\e)$ is a generalized solution of \eqref{diveq}, belonging to $L_p(Q_r(\by))$.
	\end{proposition}
\begin{proof}
(i) We see from \eqref{solscaleq} that the function $u(\e)$	belongs to $C^1(Q_r(\by))$ and it satisfies the equation $\frac{\partial u}{\partial x_1}=W(\e)$ in the cube $Q_r(\by)$. Hence $\vec\Gamma(\e)$ is a classical solution of \eqref{diveq} in $Q_r(\by)$. Claim (i) is proven

(ii) We have for $W\in L_p(Q_r(\by))$, using H\"older inequality:
\begin{eqnarray*}
&&\Vert U(W)\Vert_p^p=\int_{Q_r(\by)}\Big|\int_{y_1}^{x_1}W(s,\e^\prime)\,\mathrm{d}s\Big|^p\,\mathrm{d}\e\le\\
&&\int_{Q_r(\by)}(x_1-y_1)^{p-1}\int_{y_1}^{x_1}|W(s,\e^\prime)|^p\,\mathrm{d}s\,\mathrm{d}x_1\,\mathrm{d}\e^\prime\le \frac{r^p}{p}\Vert W\Vert_p^p. 
\end{eqnarray*}
Claim (ii) is proven.

(iii)	By claim
	(i), for any $W\in C^1(Q_r(\by))$ the vector field $\vec\Gamma(\e)$
	is a classical solution of equation (\ref{diveq}), hence it is a
	generalized solution of it, i.e. the relation (\ref{defgensol}) is
	valid. On the other hand, in view of \eqref{frmsol}, \eqref{constvecfield}, \eqref{solscaleq} and  claim (ii) ,  the correspondence
	$W\rightarrow \vec\Gamma$ is continuous with respect to $L_p$-norm.  Using the density of $C^1(Q_r(\by))$ in $L_p(Q_r(\by))$, we can extend the relation (\ref{defgensol}) by
	continuity from $W\in C^1(Q_r(\by))$ to $W\in L_p(Q_r(\by))$. Claim
	(iii) is proven.
\end{proof}

\subsection{Periodic potential solutions}

Let us look for solutions of the divergence equation \eqref{diveq} in the cube $Q_1(\vec 0)$, which satisfy the periodic boundary condition (\cite{Bour-Br}). This means that we consider these solutions $\vec \Gamma(\e)$ defined on the whole $\R^d$, which are $1$-periodic by all the variables $x_1,\,x_2,\dots,\,x_d$, i.e., they can be considered as defined on the torus $\T^d$. We shall consider  potential solutions of \eqref{diveq}, i.e., those $\vec \Gamma(\e)$, for which there is a scalar function 
$\phi(\e)$ such that $\vec \Gamma(\e)=\nabla\phi(\e)$. Then this function should satisfy the Poisson equation
\begin{equation}\label{Poiseq}
\Delta\phi=W.
\end{equation}
Recall that Fourier operator $\mathcal{F}$ on $\T^d$ maps isometrically the space $L_2(\T^d)$ onto the space $l_2(\Z^d)$ \cite{Kam}.
Denote by $L_p^\#(\T^d)$ the set of all functions $f\in L_p(\T^d)$, for which $\int_{\T^d}f(\e)\,\mathrm{d}\e=0$, i.e., $\mathcal{F}(f)(\vec 0)=0$. 
Also denote $l_p^0(\Z^d):=\{g\in l_p(\Z^d):\;g(\vec 0)=0\}$.
We have the following claim:
\begin{proposition}\label{prdiveq}
	Suppose that $p>2$, $1/p+1/q=1$, $W\in L_2^\#(\T^d)$ and
\begin{equation}\label{inclvecG}
\vec G\in l_{q}\big(\Z^d\big), 
\end{equation}	
where
\begin{equation}\label{dfvecG}
\vec G(\vec k)=\frac{\mathcal{F}(W)(\vec k)}{2\pi i|\vec k|^2}\vec k.
\end{equation} 
	Then the vector function
\begin{equation}\label{fourtransdiveq}
\vec \Gamma(\e)=\mathcal{F}^{-1}(\vec G)(\e)
\end{equation}	
belongs to $L_p^\#(\T^d)$ and it is a potential solution of divergence equation \eqref{diveq}. Furthermore, 
\begin{equation}\label{estsoldiveq}
|\Vert\vec\Gamma\Vert_p\le d\Vert\vec G\Vert_{l_{q}} .
\end{equation}
\end{proposition}
\begin{proof}
For brevity we shall omit $(\T^d)$ and $(\Z^d)$ in the notations $L_p(\T^d)$, $L_p^\#(\T^d)$, $l_q(\Z^d)$ and $l_q^0(\Z^d)$.
	Notice that since the operator $\mathcal{F}^{-1}:\,l_2\rightarrow L_2$ is isometric,
	for any $g\in l_2\;$ $\Vert\mathcal{F}^{-1}(g)_2=\Vert g \Vert_{l_2}$. On the other hand, in view of the formula for $\mathcal{F}^{-1}$, given in Section \ref{sec:notation},  $\mathcal{F}^{-1}$ maps  $l_1$ into  $L_\infty$ and for any $g\in l_1$ $\Vert\mathcal{F}^{-1}(g)\Vert_\infty\le\Vert g \Vert_{l_1}$. Then by the Riesz-Thorin Interpolation Theorem (\cite{Gr}, Theorem 1.3.4) , $\mathcal{F}^{-1}$ maps $l_q$ into  $L_p$ and for any $g\in l_q$  $\Vert\mathcal{F}^{-1}(g)\Vert_p\le\Vert g \Vert_{l_q}$.  Notice that since $p>2$,  
	$l_q\subset l_2$ and $L_p\subset L_2$. Then since $\mathcal{F}^{-1}:\,l_2\rightarrow L_2$ is injective, $\mathcal{F}^{-1}:\,l_q\rightarrow L_p$ is injective too.  It is clear that $\mathcal{F}^{-1}$ maps injectively $l_q^0$ into $L_p^\#$,   It is easy to see that the analogous property is valid for vector functions: $\mathcal{F}^{-1}$ maps injectively $l_q^0$ into $L_p^\#$ and  for any $\vec H\in l_q^0$
	\begin{equation}\label{RieszThor}
	\Vert\mathcal{F}^{-1}(\vec H)\Vert_p\le d\Vert\vec H\Vert_{l_q}. 
	\end{equation}
 Applying the Fourier transform to equation \eqref{Poiseq} and denoting $u(\vec k)=\mathcal{F}(\phi)(\vec k)$, we obtain: 
\begin{equation*}
(2\pi i)^2|\vec k|^2u(\vec k)=\mathcal{F}(W)(\vec k).
\end{equation*}
	Since $W\in L_2^\#$ (hence $\mathcal{F}(W)\in l_2$ and $\mathcal{F}(W)(\vec 0)=0$ ), the last equation has the  solution 
\begin{eqnarray*}
u(\vec k)=\frac{\mathcal{F}(W)(\vec k)}{(2\pi i)^2|\vec k|^2}
.\end{eqnarray*}
this solution belongs to the subspace $l_2^0$ and it is unique there. Hence the function $\phi=\mathcal{F}^{-1}(u)$ is a unique solution of \eqref{Poiseq} in the subspace $L_2^\#$. Moreover, since $|\vec k|^2\ u\in l_2$, this solution belongs to the Sobolev space $W_2^2(\T^d)$. Then $\vec\Gamma=\nabla\phi$ is a solution of divergence equation \eqref{diveq} and it is expressed by \eqref{fourtransdiveq}. Furthermore,    
in view of \eqref{dfvecG},  \eqref{inclvecG} and \eqref{RieszThor}, inclusion $\vec\Gamma\in L_p^\#$ and estimate \eqref{estsoldiveq} are valid.   
\end{proof} 

\section{Some estimates}
\label{sec:B}

Let $\Omega$ be a bounded open domain in $\R^d$.   
\begin{proposition}\label{lmest}
 Suppose that $d>2$.
	
\noindent	(i) If $W\in L_{d/2}(\Omega)$, then
	for any $u\in C_0^\infty(\Omega)$
	\begin{equation}\label{Sobineq}
	|\int_\Omega W(\e)|u(\e)|^2\, \mathrm{d}\e|\le C^2(d)\Vert
	W\Vert_{d/2}\int_\Omega|\nabla u(\e)|^2\, \mathrm{d}\e,
	\end{equation}
	where $C(d)$ is expressed by (\ref{dfC}),
	\vskip2mm
	
\noindent	(ii)  Suppose that there is a vector field
	$\vec\Gamma(\e)$ on $\Omega$ belonging to $L_d(\Omega)$ and
	satisfying the equation $\mathrm{div}\vec\Gamma(\e)=W(\e)$ in the generalized sense, where
	$W\in L_1(\Omega)$. Then for any $u\in C_0^\infty(\Omega)$
	\begin{equation}\label{estintlor1}
	|\int_{\Omega} W(\e)|u(\e)|^2\,\mathrm{d}\e|\le 2\,C(d)\,\Vert
	\vec\Gamma\Vert_d\int_{\Omega}|\nabla u(\e)|^2\, \mathrm{d}\e;
	\end{equation}
	\end{proposition}
\begin{proof} 
(i) Using H\"older inequality, we get:
	\begin{equation*}
	|\int_\Omega W(\e)|u(\e)|^2\, \mathrm{d}\e|\le\Vert
	W\Vert_{d/2}\Vert u\Vert_{q}^2,
	\end{equation*}
	where $q=\frac{2d}{d-2}$. On the other hand, by the Sobolev's
	theorem, the space $W_2^1(\R^d)$ is embedded continuously into
	$L_q(\R^d)$. Hence continuing each function $u\in
	C_0^\infty(\Omega)$ by zero from $\Omega$ to the whole $\R^d$, we
	get:
	\begin{equation}\label{sobineq}
	\Vert u\Vert_{q}^2\le C^2(d)\int_\Omega|\nabla u(\e)|^2\,
	\mathrm{d}\e,
	\end{equation}
	(\cite{Aub}, \cite{Tal}).  These circumstances imply the
	desired claim.
	
	(ii) Making integration by parts, we get for $u\in C_0^\infty(\Omega)$:
	\begin{equation*}
	\int_{\Omega}W(\e)|u(\e)|^2\, \mathrm{d}\e=
	-2\int_{\Omega}\Re\big(\langle\vec\Gamma\e),\,\overline{u(\e)}\,\nabla
	u(\e)\rangle\big)\, \mathrm{d}\e.
\end{equation*}
	Using Schwartz's inequality, we obtain:
	\begin{equation*}
	\int_{\Omega}W(\e)|u(\e)|^2\, \mathrm{d}\e\le\\
	2\,\Vert\vec\Gamma u\Vert_2 \Vert\nabla u|\Vert_2.
	\end{equation*}
	On the other hand, by claim (i): 
	\begin{equation*}
	\Vert\vec\Gamma u\Vert_2\le C(d)\Big(\Vert |\vec\Gamma|^2\Vert_{d/2}\Big)^{1/2}\Vert\nabla u|\Vert_2.=
	C(d)\Vert\vec\Gamma\Vert_d\Vert\nabla u|\Vert_2.
	\end{equation*}
	Te last two estimates imply the desired inequality \eqref{estintlor1}. Claim (ii) is proven.
\end{proof}


\begin{thebibliography}{ABC2}
	
	\bibitem{Aub}
	T.\,Aubin, \textit{Some nonlinear problems in Riemannian geometry},
	Springer, New-York, 1998.

\bibitem{Ben-Fort} V. Benci and D. Fortunato, \textit{Discreteness  Conditions of the Spectrum of Schr\"odinger
	Operators,} J. Math. Anal. Appl., \textbf{64} (1978), 695-700.

\bibitem{Bour-Br} J. Bourgain and H. Brezis, \textit{On the Equation $\mathrm{div}\,\mathrm{Y}=\mathrm{f}$ and Application to Control of Phases,} J. Amer.Math. Soc., Vol.\textbf{16}, No \textbf{2} (2003), 393-426

\bibitem{Ch} G. Choquet, \textit{Theory of Capacities,}
Ann. Inst. Fourier, \textbf{5} (1953), 131-295.

\bibitem{GMD} Gian Maria Dall'Ara, \textit{Discreteness of the spectrum of Schr\"odinger operators with matrix-valued non-negative
	potentials,} Journal of Functional Analysis \textbf{268} (2015),
3649-3679.

\bibitem{Gl} I.M. Glazman, \textit{Direct Methods of Qualitative Spectral Analysis of Singular Differential Operators,} Israel Program for Scientific Translations, 1965.

\bibitem{Gr} L. Grafakos, \textit{Classical Fourier Analysis,} Third Edition, Springer New York, 2014.

	\bibitem{Iwas} A. Iwasaka, \textit{Magnetic Schr\"odinger operator with compact resolvent,} J. Math. Kyoto Univ.(JMKYAZ), \textbf{26}-3 (1986), 357-374.
	
	\bibitem{Kam} D. W. Kammler, \textit{A First Course in Fourier Analysis,} Prentice Hall, Upper Saddle River, New Jersey, 2000.

	\bibitem{K-Sh} V. Kondratiev and M. Shubin, \textit{Discreteness of spectrum for the magnetic
		Schr¨odinger operators,} Commun. Partial Differential Equations
	\textbf{27} (2002), 477-525.

 \bibitem{Kon-Shub} V. Kondratiev and M. Shubin \textit{Discreteness of spectrum for the Schr\"odinger operator on manifolds with bounded geometry,} Operator Theory: Advances and Applications, Birkh\"auser Verlag Basel/Switzerland, Vol. \textbf{110}  (1999), 185-226.

\bibitem{L-S-W} D. Lenz, P. Stollmann and D. Wingert, \textit{Compactness of Schr\"odinger semigroups,} Math. Nachr. \textbf{283} (2010), No 1, 94-103.

\bibitem{Maz1} V. Mazya, \textit{Sobolev Spaces: with Applications to Elliptic Partial Differential Equations,} Springer, 2nd edition, series: Grundlehren der mathematischen Wissenschaften 342, 2011.

\bibitem{Maz} V. Mazya, \textit{Analytic criteria in the qualitative spectral analysis of the Schr\"odinger
	operator,} Procieedings of Simposia in Pure Mathematics, Spectral Theory and Mathematical Physics: A Festschrift in Honor of Barry Simon's 60th Birthday, A.M.S., Providence, Rhode Island, Vol {\bf 76}, Part {\bf 1} (2006), 257-288.

\bibitem{M-Sh} V. Mazya and M. Shubin, \textit{Discreteness of spectrum and positivity criteria for Schr\"odinger operators,} Ann. Math., \textbf{162} (2005), 919-942.

	\bibitem{M-V} V.G. Mazya and I.E. Verbitsky, \textit{The Schr\"odinger operator on the energy space: boundedness and
		compactness criteria.} Acta Math., \textbf{188} (2002), 263-302.

\bibitem{Mol} A. M. Molchanov, \textit{On conditions of discreteness of the spectrum for self-adjoint
	di®erential equations of the second order.} Trudy Mosk. Matem. Obshchestva (Proc. Moscow Math. Society)
\textbf{2} (1953), 169-199 (Russian).

\bibitem{N-W}[N-W] G.L. Nemhauser and L.A. Wolsey, \textit{Integer and Combinatorial Optimization,} John Willey\,\&\,Sons, 1988.

	\bibitem{Pau-Pf} Thierry De Pauw and Washek F. Pfeffer, \textit{Distributions for which
		$\mathrm{div}\,v=F$ has a continuous solution,} Comm. Pure Appl.
	Math. \textbf{61} (2008), No 2, 230-260.

\bibitem{Si1} B. Simon,  \textit{Schr\"odinger operators with purely discrete
	spectrum,}  Methods of Functional Analysis and Topology, Vol.
\textbf{15} (2009), no. 1, 61-66.

\bibitem{T}  M. Taylor \textit{Scattewring Lengtth and the Spectrum of $-\Delta+V$,} Canad. Math. Bull. Vol \textbf{49} (1) (2006), 144-151.

\bibitem{Zel1} L. Zelenko, \textit{Conditions of discreteness of the spectrum for Schr\"odinger operator and some optimization problems for capacity and measures,} Applied Analysis and Optimization, Vol. \textbf{3}, No 2 (2019), 281-306.

\bibitem{YRMR} A.K. Yadav, R. Ranjan, U. Mahbub and M,C. Rotkowitz, \textit{New Methods for Handling Binary Problems,} Annual Allerton Conference on Communication, Control and Computing, Allerton (2016)-Preprint

.
\bibitem{Tal} G. Talenti, \textit{Best Constant in Sobolev
		Inequality} Ann. Mat. Pura Appl. \textbf{110} (1976), 353-372.

	\end{thebibliography}
\end{document}